\documentclass[psamsfonts]{amsart}

\usepackage{amsfonts}

\usepackage{amsmath,amstext,amssymb,amsopn,amsthm}
\usepackage{amsmath,amssymb,amsthm}
\usepackage[mathscr]{eucal}
\usepackage{amssymb}
\usepackage{amsmath}
\usepackage{color}

\definecolor{mahogany}{cmyk}{0, 0.77, 0.87, 0}
\definecolor{salmon}{cmyk}{0, 0.53, 0.38, 0}
\definecolor{melon}{cmyk}{0, 0.46, 0.50, 0}
\definecolor{yellowgreen}{cmyk}{0.44, 0, 0.74, 0}
\definecolor{brickred}{cmyk}{0, 0.89, 0.94, 0.28}
\definecolor{OliveGreen}{cmyk}{0.64, 0, 0.95, 0.40}
\definecolor{RawSienna}{cmyk}{0, 0.72, 1.0, 0.45}
\definecolor{ZurichRed}{rgb}{1, 0, 0} 

\usepackage{tikz}
\usetikzlibrary{arrows}

\usepackage[unicode,bookmarks,colorlinks]{hyperref}
\hypersetup{
    linkcolor=brickred,
}

\numberwithin{equation}{section}

\newtheorem{thm}{Theorem}[section]        
\newtheorem{cor}{Corollary}[section]
\newtheorem{lem}{Lemma}[section]
\newtheorem{prop}{Proposition}[section]
\newtheorem{rmk}{Remark}[section]
\newtheorem{example}{Example}[section]

\newcommand{\R}{\mathbb{R}}                  
\newcommand{\Rd}{\R^d}               
\newcommand{\ioRd}{\int_{\Rd}}              
\newcommand{\iogRd}[1]{\int_{\R^{#1d}} }    
\newcommand{\lgint}[2]{ \int_{0}^{1} \int_{0}^{\lambda_1}\cdots \int_{0}^{ \lambda_{#1}}\int_{\R^{#2d}}  } 
\newcommand{\set}[1]{ \left\{#1\right\} }
\newcommand{\mysum}[3]{\sum\limits_{#1=#2}^{#3}}          
\newcommand{\myprod}[3]{\prod\limits_{#1=#2}^{#3}}

\newcommand{\texpt}[1]{E\left[#1\right]}
\newcommand{\wh}[1]{\widehat{#1}}              
\newcommand{\pV}[1]{\myprod{i}{1}{#1}\wh{V}(\theta_i)}
\newcommand{\sV}[1]{\wh{V}(-\mysum{i}{1}{#1}\theta_i)}
\newcommand{\opV}{\sV{j-1}\pV{j-1}}
\newcommand{\eid}{\,\,{\buildrel \mathcal{D} \over =}\,\,}   
\newcommand{\trace}{Tr( e^{-tH_{V}} - e^{-tH_{\alpha}})}
\newcommand{\Fjalp}[2]{F_{#1}^{(\alpha)}(#2,\xi,\theta)}
\newcommand{\Lj}{L_{j}^{(\alpha)}(\lambda,\theta)}
\newcommand{\calS}{{\mathcal{S}} }
\newcommand{\incmt}[2]{\lambda_{#1}-\lambda_{#2}}        
\newcommand{\abs}[1]{\left|#1\right|}

\newcommand{\spS}[1]{S_{#1}^*}
\newcommand{\alpf}[1]{#1_{\alpha}}
\newcommand{\pH}{ p_{t}^{H_{V}}}
\newcommand{\palp}{p_{t}^{(\alpha)}}
\newcommand{\tgo}{t\downarrow 0}
\newcommand{\const}{ C_{d,\alpha} }
\newcommand{\fract} { (-\Delta)^{ \frac{\alpha}{2} } }

\begin{document}

\title{trace asymptotics for fractional Schr\"{o}dinger Operators}
\author{Luis Acu\~na Valverde}\thanks{Supported in part  by NSF Grant
\# 0603701-DMS under PI Rodrigo Ba\~nuelos}
\address{Department of Mathematics, Purdue University, West Lafayette, IN 47907, USA}
\email{lacunava@math.purdue.edu}
\maketitle
\begin{abstract}
This paper proves an analogue of a result of Ba\~{n}uelos and S\'{a} Barreto \cite{Ba.Sab}  on the asymptotic expansion for the trace of Schr\"{o}dinger operators on $\R^d$ when the Laplacian $\Delta$, which is the generator of the Brownian motion, is replaced by the non-local integral operator $\Delta^{\alpha/2}$, $0<\alpha<2$, which is the generator of the symmetric stable process of order $\alpha$. These results also extend recent results of Ba\~{n}uelos and Yildirim \cite{Ba.Sel} where the first two coefficients for  $\Delta^{\alpha/2}$ are computed.  Some extensions to Schr\"odinger operators arising from relativistic stable and mixed stable processes are obtained. 
\end{abstract}

\tableofcontents

\section{Introduction.}
Heat asymptotic results have been widely used in areas of spectral theory and its applications to scattering theory, statistical and quantum mechanics and in several areas in geometry.  We refer the reader to van den Berg \cite{vanden1} for the computation of the first two terms in the asymptotic expansion of the trace of the heat kernel of the Schr\"{o}dinger operator $-\Delta+V$ under H\"older continuity of the potential and to Ba\~{n}uelos and  S\'{a} Barreto \cite{Ba.Sab} for a more general computation with an explicit formula for all the coefficients for potentials $V\in \calS(\Rd)$, the class of rapidly decaying functions at infinity, and for applications to scattering theory.  For applications in statistical mechanics and quantum theory, we refer the reader to the articles of Lieb 
\cite{Lieb} and  Penrose and Stell \cite{Penrose} about the second viral coefficient of a hard--sphere gas at low temperature and sticky spheres, respectively. Heat trace asymptotic for the Laplacian have been of interest for many years for domains in Euclidean space $\Rd$ and on manifolds where the coefficients reveal many geometric quantities such as volume, surface area, convexity, number of holes, etc.  For more on this large literature as well as some historical perspective, we refer the reader to Arndt, Nittka, Peter and Steiner, \cite[pp 1-71]{AreSch}, Ba\~{n}uelos, Kulczycki and Siudeja \cite{Ba.Kul.Siu,Ba.Kul}, Datchev and Hezari's \cite{KirHez}, Donelly \cite{Don}, McKean and Moerbeke \cite{MckMoe}, and Colin De Verdi\`ere \cite{dev}.

Let $H_2=-\Delta$ and $H_V=-\Delta+V$, $V\in \calS(\Rd)$.   In \cite{Ba.Sab}, the existence of an asymptotic expansion of the trace of the operator $e^{-tH_V} - e^{-tH_2}$,  as $\tgo$, is proved.  To make the connection to the fractional Laplacian more clear, let us denote the heat kernel for $-\Delta$ by 
\begin{equation*}\label{heat2}
p^{(2)}_t(x)=\frac{1}{{(4\pi t)}^{d/2}}e^{-\frac{|x|^2}{4t}},
\end{equation*}
so that 
\begin{equation*}
p^{(2)}_t(0)=\frac{1}{{(4\pi t)}^{d/2}}.
\end{equation*}
Set 
\begin{equation}\label{lambdas}
I_{j}=\set{\lambda=(\lambda_1,...,\lambda_{j}): 0<\lambda_{j}<\lambda_{j-1}<...<\lambda_1<1}.
\end{equation}
 Throughout the paper we use the notation $f(t)= {\mathcal O}(g(t))$, as $\tgo$, to mean that  there exist constants $C$ and $\delta$ such that $|f(t)|\leq C|g(t)|$, for $0<t<\delta$.

With this notation the result in \cite{Ba.Sab} can be stated as follows.  For any integer $J\geq1$,
\begin{equation}\label{Ba.Sab trace}
\frac{Tr(e^{-tH_V} - e^{-tH_{2}})}{ p^{(2)}_t(0)}=\mysum{l}{1}{J}c_{l}(V)t^l + {\mathcal{O}}(t^{J+1}),
\end{equation}
as $\tgo$, with 
\begin{align*}
c_1(V)&=-\ioRd V(\theta)d\theta,  \,\,\, c_{l}(V)= (-1)^l \sum_{\substack{j+n=l\\j\geq2}}C_{n,j}^{(2)}(V), \,\,\, C_{d,2}=(2\pi)^d, \nonumber \\
C_{n,j}^{(2)}(V)&= \frac{C_{d,2}}{ (2\pi)^{jd} n!}\int_{I_{j}}\iogRd{(j-1)}\set{L_{j}^{(2)}(\lambda,\theta)}^n 
\opV d\theta_i d\lambda_i d\lambda_j, \,\textrm{and} \,\,\, \nonumber \\
 L_{j}^{(2)}(\lambda,\theta)&=\mysum{k}{1}{j-1}(\incmt{k}{k+1})\abs{ \mysum{i}{1}{k}\theta_i}^2-
\abs{\mysum{k}{1}{j-1}(\incmt{k}{k+1})\mysum{i}{1}{k}\theta_i}^2. 
\end{align*}
In particular, for $J=2$, the formula gives
\begin{equation}\label{vandenberg}
\frac{Tr(e^{-tH_V} - e^{-tH_{2}})}{ p^{(2)}_t(0)}+t\ioRd V(\theta)d\theta - \frac{t^{2}}{2!}\ioRd V^2(\theta)d\theta={\mathcal O}(t^3), 
\end{equation}
as $\tgo$ which is the van den Berg  \cite{vanden1} results under our assumption on $V$.  For $J=3$, the formula gives 
\begin{align}\label{Ba.Sab.result}
\frac{Tr(e^{-tH_V} - e^{-tH_{2}})}{ p^{(2)}_t(0)} &+t\ioRd V(\theta)d\theta - \frac{t^{2}}{2!}\ioRd V^2(\theta)d\theta \nonumber\\ &+\frac{ t^{3}}{3!}\ioRd V^3(\theta)d\theta + \frac{t^3}{12}\int_{\Rd}|\nabla V(\theta)|^2d\theta={\mathcal O}(t^4), 
\end{align}
as $\tgo$.

For $d=1$, a recurrent formula for the general coefficients in the expansion  was obtained in the seminal paper by McKean-Moerbeke \cite{MckMoe} using KdV methods. Using these techniques, and the symmetry of certain integrals, Colin De Verdi\`ere \cite{dev} computed the first four coefficients in $\R^3$. 
The results in this paper are motivated by  \cite{Ba.Sab} where (\ref{Ba.Sab trace}) is proved by Fourier transform methods for all $d\geq 1$.  Our proof is a combination of probabilistic arguments and Fourier transform techniques and unfortunately is much more technical than  \cite{Ba.Sab}.  These results are also motivated by  \cite{Ba.Sel} where an analogue of van den Berg's results \cite{vanden1} (the computation of the first two terms) is proved for the fractional Laplacian and other related non-local operators.  It is interesting to observe here that (integration by parts)   
$$
\int_{\Rd}|\nabla V(\theta)|^2d\theta=\int_{\R^d}-\Delta V(\theta) V(\theta)d\theta = \mathcal{E}(V, V),
$$
which is the Dirichlet form of $V$ with respect to the Laplacian. Based on this, it is natural to conjecture that the third term in the expansion for the fractional Laplacian should involve the Dirichlet form of $V$ for the operator $(-\Delta)^{\alpha/2}$. But this is not the case, as we shall see momentarily, which is somewhat surprising. 

To state our results for stable processes, we briefly introduce the $\alpha/2$--subordinators in order to more clearly exhibit the similarities and differences from our result  to the Ba\~{n}uelos--S\'{a} Barreto \cite{Ba.Sab} result. For $0<\alpha<2$, an $\alpha/2$--subordinator is an almost surely  non-decreasing $[0,\infty)$-valued process $S=\set{S_{t}}_{t\geq 0}$ starting at 0 and uniquely determined by its Laplace transform 
\begin{equation*}\label{Lap.Subor}
 \texpt{e^{-\lambda S_{t}}}=e^{ -t\lambda^{\alpha/2}},
\end{equation*}
for all  $t>0$ and $\lambda>0$. Throughout this paper we will often write $S_{1,\alpha/2}$ for $S_1$ to emphasize the $\alpha$ dependence. We write $Z\eid Y$ for two random variables $Z,Y$ with values in $\Rd$ to mean that they are equal in distribution or have the same law. That is, for any Borel set $B\subset \Rd$, $P(Z\in B)=P(Y\in B)$, where $Z,Y$ could be defined on different probability spaces. Our analogue result to (\ref{Ba.Sab trace}) for the fractional Laplacian $\fract$ is provided by the following theorem.

\begin{thm}\label{mainthm}
Let $0<\alpha<2$ be given. Suppose $V\in \calS(\Rd)$ and denote the fractional Laplacian and its associated fractional Schr\"odiner operator by $\alpf{H}=\fract$ and $H_{V}=\alpf{H}+V$, respectively.   Denote the heat kernel for $\fract$ by $\palp(x)$ (see \eqref{Scal.prop} below).  Assume that $M\geq 1$ is an integer  satisfying $M<\frac{d+\alpha}{2}$. 
\begin{itemize}
\item[(a)] Given $J\geq 2$, there exists a bounded function $R_{J+1}^{(\alpha)}(t)$, $0<t<1$, such that
\begin{align}\label{myexpansion}
\frac{\trace}{\palp(0)}= -t\ioRd V(\theta)d\theta + \mysum{j}{2}{J} \mysum{n}{0}{M-1}(-1)^{n+j}C_{n,j}^{(\alpha)}(V)
t^{\frac{2n}{\alpha}+j}+t^{\Phi_{J+1}^{(\alpha)}(M)}R_{J+1}^{(\alpha)}(t)
\end{align} 
where 
$$\Phi_{J+1}^{(\alpha)}(M)=\min\set{J+1,2+\frac{2M}{\alpha}},$$ and the constants $C_{n, j}^{(\alpha)}(V)$ are given by
\begin{align*}\label{quick definition}
C_{n,j}^{(\alpha)}(V)&=\frac{\const}{ (2\pi)^{jd} n!}\int_{I_{j}}\iogRd{(j-1)} \texpt{ S_{1,\frac{\alpha}{2}}^{-d/2}\set{\Lj}^{n}}\opV d\theta_i d\lambda_i d\lambda_j,  \\
\Lj&= \mysum{k}{1}{j-1}\spS{\incmt{k}{k+1}}\abs{ \mysum{i}{1}{k}\theta_i}^2 
-\frac{1}{S_{1,\frac{\alpha}{2}}} |\mysum{k}{1}{j-1}\spS{\incmt{k}{k+1}}\mysum{i}{1}{k}\theta_i|^2, \,\, and \,\,
 \const= \frac{\pi^{d/2}}{ p_{1}^{(\alpha)}(0)}, \nonumber
\end{align*}
where the $\lambda_k's$ are as in \eqref{lambdas}. Moreover, the random variables 
$\spS{\incmt{1}{2}}$, $\spS{\incmt{2}{3}}$,...,$\spS{\incmt{j-1}{j}}$, $\spS{1-(\incmt{1}{j})}$ are  independent and satisfy
\begin{align*}
\spS {1-(\incmt{1}{j})} + \mysum{k}{1}{j-1}\spS{\incmt{k}{k+1}}&=S_{1,\frac{\alpha}{2}}
 \end{align*}
and $\spS{l}\eid S_{l}$ for any $l\in\set{{1-(\incmt{1}{j}),\incmt{k}{k+1}}}_{k=1}^{j-1}$.\\
\item[(b)] For any $j\geq 2$ and $1\leq n \leq M$, $$\lim\limits_{\alpha\uparrow2}C_{n,j}^{(\alpha)}(V)=C_{n,j}^{(2)}(V).$$
\end{itemize}
\end{thm}
\bigskip

We note that when $\alpha=2$ the last Theorem remains true and $\spS{\incmt{k}{k+1}}=\lambda_k-\lambda_{k+1}$ and the condition on $d$ and $M$ is not needed. The reason for this is that in the later case, $S_t=t$ and  then $S_{1,1}=1$.  What part (b) in the theorem proves is that our results are robust. 

To see the connection to the Ba\~{n}uelos and S\'{a} Barreto result more clearly, we state the following theorem which is an immediate  consequence of Theorem \ref{mainthm} (see \S\ref{sec:proof.theorem}) and which resembles (\ref{Ba.Sab trace}) more closely. 
 
\begin{thm}\label{2.mainthm}
Under the same conditions of Theorem \ref{mainthm}, we have
\begin{eqnarray}\label{resembleformula}
\frac{\trace}{\palp(0)} &=&-t\ioRd V(\theta)d\theta\,\,\,+ \sum\limits_{\substack {\frac{2n}{\alpha}+j<\Phi_{J+1}^{(\alpha)}(M) \\ 2\leq j\leq J,\,\, 0\leq n\leq M-1}}(-1)^{n+j}C_{n,j}^{(\alpha)}(V)t^{\frac{2n}{\alpha}+j}\\
&+&\mathcal{O}(t^{\Phi_{J+1}^{(\alpha)}(M)}),\nonumber
\end{eqnarray}
as $\tgo$. 
\end{thm}

 Now, to obtain  (\ref{Ba.Sab trace}) from the last theorem we note again that for $\alpha=2$ we have no restrictions on $J$ and $M$ other than $J\geq 2$ and $M\geq1$. Also observe that $\Phi_{J+1}^{(2)}(M)=\min\set{J+1,M+2}$. Then, by taking $M=J-1$ we conclude $\Phi_{J+1}^{(2)}(J-1)=J+1$. As a consequence of \eqref{resembleformula}, we arrive at 
\begin{equation*}
\frac{Tr(e^{-tH_V} - e^{-tH_{2}})}{ p^{(2)}_t(0)}= -t\ioRd V(\theta)d\theta\,\,\, +\sum\limits_{\substack {n+j<J+1 \\ 2\leq j\leq J,\,\, 0\leq n\leq J-2}}(-1)^{n+j}C_{n,j}^{(2)}(V)t^{n+j}+\mathcal{O}(t^{J+1}),
\end{equation*}
as $\tgo$.
But, notice that in this case, 
\begin{equation*}
\sum\limits_{\substack {n+j<J+1 \\ 2\leq j\leq J, 0\leq n\leq J-2}}(-1)^{n+j}C_{n,j}^{(2)}(V)t^{n+j}=\mysum{l}{2}{J}c_l(V)t^l, 
\end{equation*}
and (\ref{Ba.Sab trace}) follows.

In \S \ref{sec:Particular.expansions} we provide more specific expansion formulas 
for $\alpha$ \textquoteright s of the form $2/k$, where $k$ positive integer. These examples are the only cases  where  $\frac{2n}{\alpha}+j$ are integers for all $n,j$, because for the particular case $n=1$ and $j=2$ there exists an integer $m_0\geq3$ such that $\frac{2}{\alpha}+2=m_0$, which implies that $\alpha=\frac{2}{m_0-2}$.
 
The assumption $M<\frac{d+\alpha}{2}$ in our theorem is sufficient to prove two crucial facts needed in our expansion.  Namely, (1) that the coefficients in  Theorem \ref{mainthm} are finite  and (2) that the remainders that appear in the definition of $R_{J+1}^{(\alpha)}(t)$ (see \S \ref{sec:boundedcoeff.rem} below) are bounded for $t\in (0,1)$. 
Since $M\geq 1$, the condition $M<\frac{d+\alpha}{2}$ determines, for a given $d$, the range of $\alpha$'s  for which Theorem \ref{mainthm} holds. Thus, for example when $M=1$ and $d=1$, Theorem \ref{mainthm} only permits the range $1<\alpha<2$.  In \S \ref {sec:improvement} we will show how a modified version of this condition (namely $\frac{M}{2}-\frac{d}{4}<\frac{\alpha}{2}$) can widen the range of $\alpha$'s  for which Theorem \eqref{mainthm} remains true when $d=1,2,3$ and $M=1,2$.
 
A particular case of Theorem \ref{mainthm} and our results in \S \ref{sec:improvement} is the following corollary which extends the results in \cite{Ba.Sel} where the second coefficient is computed.
\begin{cor}\label{maincor}
\hspace*{20mm}

\bigskip
{\bf (i)} For $d=1$ ,
\begin{eqnarray}
\frac {\trace}{\palp(0)} + t\int_{\R} V(\theta)d\theta - \frac{t^{2}}{2!}\int_{\R} V^2(\theta)d\theta + \frac{t^3}{3!}\int_{\R} V^3(\theta)d\theta &=& \nonumber \\
\left\{\begin{array}{cc}
       {\mathcal O}( t^{2 + \frac{2}{\alpha}}), &  \mbox{ if $\alpha\in(1,2)$}, \nonumber\\
       {\mathcal O}(t^4), & \mbox{ if $\alpha\in(1/2,1]$}, 
       \end{array}   
\right. 
\end{eqnarray} 
as $\tgo$.
\bigskip

{\bf (ii)} For $d=1$ and $\frac{3}{2}<\alpha<2$, we have
\begin{align*}
\frac{\trace}{p_t^{(\alpha)}(0)} + &t\int_{\R} V(\theta)d\theta - \frac{t^{2}}{2!}\int_{\R} V^2(\theta)d\theta  \nonumber \\
+ &\frac{t^{3}}{3!}\int_{\R} V^3(\theta)d\theta +{\mathcal L}_{1,\alpha}t^{2+\frac{2}{\alpha}}\int_{\R}|\nabla V(\theta)|^2\,d\theta={\mathcal O}(t^4),
\end{align*}
as $\tgo$.
\bigskip

{\bf (iii)} For $d\geq2$,
\begin{eqnarray}
\frac {\trace}{\palp(0)} + t\ioRd V(\theta)d\theta - \frac{t^{2}}{2!}\ioRd V^2(\theta)d\theta + \frac{t^3}{3!}\ioRd V^3(\theta)d\theta &=& \nonumber \\
\left\{\begin{array}{cc}
       {\mathcal O}( t^{2 + \frac{2}{\alpha}}), &  \mbox{ if $\alpha\in(1,2)$}, \nonumber\\
       {\mathcal O}(t^4), & \mbox{ if $\alpha\in(0,1]$}, 
       \end{array}   
\right. 
\end{eqnarray} 
as $\tgo$.
\bigskip

{\bf (iv)} For $d\geq2$ and $1<\alpha<2$,
\begin{align*}
\frac {\trace}{\palp(0)} +&t\ioRd V(\theta)d\theta - \frac{t^{2}}{2!}\ioRd V^2(\theta)d\theta \nonumber \\
+ &\frac{t^3}{3!}\ioRd V^3(\theta)d\theta 
+{\mathcal L}_{d,\alpha}\,t^{2+\frac{2}{\alpha}}\,\int_{\Rd}|\nabla V(\theta)|^2\,d\theta  ={\mathcal O}(t^4),
\end{align*} 
as $\tgo$.
\bigskip

{\bf (v)} For $d\geq 3$, $\frac{2}{3} < \alpha\leq1$, 
\begin{eqnarray*}
\frac{\trace}{p_t^{(\alpha)}(0)} &+& t\ioRd V(\theta)d\theta - \frac{t^{2}}{2!}\ioRd V^2(\theta)d\theta +
\frac{t^{3}}{3!}\ioRd V^3(\theta)d\theta\\
&-&\frac{t^{4}}{4!}\ioRd V^4(\theta)d\theta +{\mathcal L}_{d,\alpha}\,t^{2+\frac{2}{\alpha}} \,\int_{\Rd}|\nabla V(\theta)|^2\,d\theta 
={\mathcal O}(t^5),
\end{eqnarray*}
as $\tgo$. 

Also for $d\geq 3$ and $\frac{1}{2}\leq\alpha\leq\frac{2}{3}$,
\begin{eqnarray*}
\frac{\trace}{p_t^{(\alpha)}(0)} &+& t\ioRd V(\theta)d\theta - \frac{t^{2}}{2!}\ioRd V^2(\theta)d\theta +
\frac{t^{3}}{3!}\ioRd V^3(\theta)d\theta\\
&-&\frac{t^{4}}{4!}\ioRd V^4(\theta)d\theta
={\mathcal O}(t^5),
\end{eqnarray*}
as $\tgo$.
\bigskip

{\bf (vi)} 
For $d\geq4$ and $0<\alpha<\frac{1}{2}$,
\begin{eqnarray*}
\frac{\trace}{p_t^{(\alpha)}(0)} &+& t\ioRd V(\theta)d\theta - \frac{t^{2}}{2!}\ioRd V^2(\theta)d\theta +
\frac{t^{3}}{3!}\ioRd V^3(\theta)d\theta\\
&-&\frac{t^{4}}{4!}\ioRd V^4(\theta)d\theta
={\mathcal O}(t^5),
\end{eqnarray*}
as $\tgo$.
\bigskip

The constants ${\mathcal{L}}_{d,\alpha}$ are defined as follows:
$${\mathcal{L}}_{d,\alpha} =\frac{ \const K_1(d,\alpha)}{(2\pi)^{d}},\,\,\,\,\, \const= \frac{\pi^{d/2}}{ p_{1}^{(\alpha)}(0)},$$
 with
 \begin{align*}
K_1(d,\alpha)= \int_0^1 \int_0^{\lambda_1} \texpt{\frac{ \spS{1-w} \spS{w}} { 
    ( \spS{1-w} + \spS{w} )^{  1+\frac{d}{2} } } }dw d\lambda_1.
\end{align*}
\end{cor}

 The question of whether our result holds regardless of the choice of $d$ and $M$ as in (\ref{Ba.Sab trace}) remains an interesting open problem which reduces to verifying that the expectations in the formula for $C_{n,j}^{(\alpha)}(V)$ are finite for all $n$ and $d$. 

 To gain a better understanding of the applications of the robustness result (part $(b)$)  in Theorem \ref{mainthm}, which is  proved by means of weak convergence, consider the following special case of Corollary \ref{maincor}. For all $d\geq1$ and $\frac{3}{2}<\alpha<2$, we have 
\begin{align*}
\frac {\trace}{\palp(0)} +&t\ioRd V(\theta)d\theta - \frac{t^{2}}{2!}\ioRd V^2(\theta)d\theta \nonumber \\
+ &\frac{t^3}{3!}\ioRd V^3(\theta)d\theta 
+{\mathcal L}_{d,\alpha}\,t^{2+\frac{2}{\alpha}}\, \int_{\Rd}|\nabla V(\theta)|^2\,d\theta   ={\mathcal O}(t^4),
\end{align*} 
as $\tgo$.  
Interestingly, due to part (b)  we see that ${\mathcal L}_{d,\alpha}\rightarrow\frac{1}{12}$ as $\alpha\uparrow2$, despite of the fact that  thus  far we are only able to provide a representation which enables us to conclude that ${\mathcal L}_{d,\alpha}$ are finite and strictly positive with no other explicit knowledge for this quantity.

The rest of this paper is organized as follows. In \S\ref{sec:stable-sub}, we introduce notation, definitions and probabilistic facts about stable processes and subordinators. In \S\ref{sec:trace&Fourier}, we find  formulas (and bounds) for the difference of the Fourier transform of the heat kernels $\wh{\pH}(\xi,\eta)-\wh{\palp}(\xi,\eta)$ that will allow us to express the trace in terms of Fourier transforms. In \S\ref{sec:boundedcoeff.rem}, we prove the boundedness of the (J+1)--th term  in the trace formula found in \S\ref{sec:trace&Fourier} for $t\in (0,1)$.
In \S\ref{sec:trace-sub}, we simplify the trace formula given in \S\ref{sec:trace&Fourier} by finding an explicit value of certain integrals in terms of the subordinator process. In \S\ref{sec:coeffic.sect}, we use an elementary Taylor expansion formula for the exponential function to define the coefficients and remainders involved in Theorem \ref{mainthm}. In \S\ref{sec:improvement}, we give an improvement of Theorem \ref{mainthm} when $d=1,2,3$. The proof of Theorem \ref{mainthm} is given in 
\S\ref{sec:proof.theorem}. In \S\ref{sec:explicitcoeff}, we compute some coefficients and provide a representation for the constants ${\mathcal {L}}_{d,\alpha}$. In \S\ref{sec:proofcor}, we prove Corollary \ref{maincor}. Explicit expansion and examples are provided for some particular $\alpha$'s in \S \ref{sec:Particular.expansions}.
Finally, in \S\ref{sec:rel.proc} and \S \ref{sec:mixed.proc}, we extend the result i) and iii) in Corollary \ref{maincor} to $\alpha$--relativistic processes and to mixed--stable processes, respectively. 

\section{Stable Processes and subordinator.} \label{sec:stable-sub}

Let $X=\set{X_t}_{t\geq0}$ be the $d$-dimensional symmetric $\alpha$--stable process of order $\alpha \in (0,2]$. The process X has stationary independent increments, which means that if $s,t>0$, then the increment $X_{t+s}-X_{t}$ is independent of the process $(X_u,0\leq v\leq t)$ and has the same law as $X_s$. Moreover, its transition density  $\palp(x,y)=\palp(x-y)$, $t>0$, $x,y\in \Rd$, is determined by its Fourier transform (characteristic function) as 

\begin{equation*}\label{CharF.Proc}
e^{-t\abs{\xi}^{\alpha}}=E^{0}[e^{-i <\xi,X_t> }]=\ioRd e^{-i<y,\xi>}\palp(y)dy,
\end{equation*}
for all $t>0$, $\xi\in\Rd$, 
where $P^{x}$ and $E^{x}$ denote the probability and expectation, respectively, of the process starting at $x$. Then, we have for any Borel set $A\subset \Rd$, 
 \[P^{x}(X_t\in A)= \int_{A} p_{t}^{(\alpha)}(x-y)dy.\]
Henceforth, $E$ will denote the expectation of both an arbitrary random variable or processes started at 0, whereas a.s will mean almost surely. 

We recall again that an $\alpha/2$--subordinator is a.s non-decreasing $[0,\infty)$-valued process $S=\set{S_{t}}_{t\geq 0}$ which  also has stationary, independent increments, starting at 0 and  uniquely determined by its Laplace Transform 

\begin{equation*}\label{Lap.Subor}
 \texpt{e^{-\lambda S_{t}}}=e^{ -t\lambda^{\alpha/2}},
\end{equation*}
for all  $t>0$ and $\lambda>0$.

Notice that the last equality implies that for all $\xi \in \Rd$,
\begin{equation}\label{CharF-LapS}
\texpt{ e^{- \abs{\xi}^2 S_{t}} }=e^{ -t \abs{\xi}^{\alpha}}.
\end{equation}
We point out that this equality is the link between the results in this paper and the results in \cite{Ba.Sab} and  it will be used several times throughout the paper. 

It is a standard fact (see \cite[p.22]{Bog}) that the $\alpha$-stable process X can be obtained as a random time change of Brownian motion where this random time is an $\alpha/2$-subordinator. In other words, we can write $X_t=B_{2S_{t}}$ where $B$ is a $d$-dimensional Brownian motion and $S$ is a $\alpha/2$-subordinator and these  are independent processes. Using this last fact, the Fourier inversion formula (see \S\ref{sec:trace&Fourier}), we see that 
\begin{align}\label{exptofS_1}
\palp(x)&=(2\pi)^{-d}\ioRd e^{i <\xi,x>}e^{-t \abs{\xi}^{\alpha}}d\xi\\
           &=\int_{0}^{\infty}(4\pi s)^{-d/2}e^{-\frac{\abs{x}^2}{4s}}\eta_{t}^{(\alpha/2)}(s)ds\nonumber\\
           &  = \texpt{p_{S_{t}}^{(2)}(x)}\nonumber,
\end{align}
where $p_{t}^{(2)}(x)=(4 \pi t)^{-d/2}e^{- \frac{|x|^2}{4t}}$ and $\eta_{t}^{(\alpha/2)}(s)$ correspond to the transition densities of the $d$-dimensional Brownian Motion process $\{ B_{2t} \}_{t>0}$ and the $\alpha/2$-subordinator, respectively.
It follows from (\ref{exptofS_1}) that $\palp(x)$ is radial, symmetric and decreasing in $x$. Moreover, these functions satisfy the following scaling property and inequality. 
 
\begin{equation}\label{Scal.prop}
\palp(x)=t^{-d/\alpha}p_{1}^{(\alpha)}(t^{-1/\alpha}x)\leq t^{-d/\alpha}p_{1}^{(\alpha)}(0),
\end{equation}
where
 \[p_{1}^{(\alpha)}(0)= \frac{w_d\Gamma(d/\alpha)}{(2\pi)^d\alpha}.\]
 Here and for the rest of the paper, $w_d$ denotes the surface area of the unit sphere in $\Rd$.
From equation (\ref{exptofS_1}), we also conclude that 
   
\begin{equation}\label{expl.expt.S1}
(4\pi)^{d/2}p_{1}^{(\alpha)}(0)=\texpt{S_{1,\frac{\alpha}{2}}^{-d/2}}.
\end{equation}
  
In fact, we claim that  for all $-\infty <\eta< \frac{\alpha}{2}$, 
\begin{equation}\label{gen.expt.S1}
\texpt{S_{1,\frac{\alpha}{2}}^{\eta}}=\frac{\Gamma(1-\frac{2\eta}{\alpha})}{\Gamma(1-\eta)}.
\end{equation}
To see this, we observe that $\left(ZS_{1,\frac{\alpha}{2}}^{-1}\right)^{\alpha/2}\eid Z$, where $Z=\exp(1)$ is an exponential random variable with parameter $1$ independent of $S_{1,\frac{\alpha}{2}}$.  By independence we have  
\begin{align*}
P( \left(ZS_{1,\frac{\alpha}{2}}^{-1}\right)^{\alpha/2}\leq \lambda)&=\int_{0}^{\infty}\left(\int_{0}^{\lambda^{2/\alpha}s}
e^{-u}du\right)\eta_1^{(\alpha/2)}(s)ds\\
&= 1-\texpt{ e^{ -\lambda^{2/\alpha} S_1 }}\\
&= 1-e^{\lambda}=P(Z\leq\lambda).
\end{align*}
Hence, it also follows by independence that
\[
\texpt{Z^{-\eta}}\texpt{S_{1,\frac{\alpha}{2}}^{\eta}}= \texpt{Z^{-\frac{\eta2}{\alpha}}},
\]
provided  $\texpt{Z^{-\eta}}$ and $\texpt{Z^{\frac{-\eta2}{\alpha}}}$ are both finite.  But, this only holds  when $-\infty <\eta< \frac{\alpha}{2}$, since 
\[ 0<\texpt{Z^{\gamma}}= \int_0^{\infty}s^{(\gamma+1)-1}e^{-s}ds=\Gamma(\gamma+1)<\infty ,\]
when $\gamma+1>0$.

The equalities (\ref{expl.expt.S1}) and (\ref{gen.expt.S1}) will be useful in proving the  finiteness for the coefficients and the boundedness of the function $R_{J+1}^{(\alpha)}(t)$ in Theorem \ref{mainthm}. We also mention that the condition given in Theorem \ref{mainthm} is derived from \eqref{gen.expt.S1}.

\section{Heat trace in terms of Fourier transform.}\label{sec:trace&Fourier}
 
Let $\wh{V}$ denote the Fourier transform of $V\in \calS(\Rd)$ with the normalization 
$$\wh{V}(\xi)= \ioRd e^{-i<x,\xi>}V(x)dx.$$
We note that because  of our definition of  $\wh{V}$, we have
\begin{enumerate}
\item[(i)] (Inversion formula)
\begin{equation*} \label{InversionForm}
V(x)=\frac{1}{(2\pi)^{d}}\ioRd e^{i<x,\xi>}\wh{V}(\xi)d\xi.
\end{equation*} 
\item[(ii)] For $f,g\in \calS (\Rd)$,
\begin{equation}\label{Plancherel}
\ioRd e^{-i<x,\xi >}f(x)g(x)dx= \frac{1}{(2\pi)^{d}}\ioRd \wh{f}(\theta)\wh{g}(\xi-\theta)d\theta.
\end{equation}
\end{enumerate}

We recall that the linear operator $\alpf{H}=(-\Delta)^{\alpha/2}$ with domain 
$$\set{f \in L^{2}(\Rd): \abs{\xi}^{\alpha}\wh{f}(\xi)\in L^{2}(\Rd)},$$
 $\alpha\in(0,2)$, is defined via Fourier transform by $\wh{\alpf{H}f}(\xi)=\abs{\xi}^{\alpha}\wh{f}(\xi)$. This operator, often referred to as the fractional Laplacian, is essentially a self-adjoint operator on $C_{0}^{\infty}(\Rd)$ with spectrum $Spec(\alpf{H})=[0,+\infty)$. For $V\in \calS(\Rd)$, we call  $H_{V}=(-\Delta)^{\frac{\alpha}{2}}+V$ the fractional Schr\"{o}dinger operator with potential V, where V acts as a multiplication operator. $H_{V}$ is self-adjoint in the domain of $\alpf{H}$ with $Spec(H_{V})=[0,+\infty)$. We write  $e^{-t\alpf{H}}$ and $e^{-tH_{V}}$ for the heat semigroups acting on $L^{2}(\Rd)$ with heat kernels given by $\palp(x,y)$ and $\pH(x,y)$, respectively. We refer the reader to  \cite{Barry,Davies} for general definition and spectral properties of the semigroup $e^{-tA}$ of the self-adjoint operator $A$ on Hilbert-space. For our purpose in this paper we recall the
 Feynman-Kac formula which gives the heat kernel for $H_{V}$.  That is, 
\begin{equation}\label{Fey.Kac.f}
\pH(x,y)=\palp(x,y)E_{x,y}^t\left[e^{-\int_{0}^{t}V(X_{s})ds}\right], 
\end{equation} 
where $E_{x,y}^t$ is the expectation with respect to the stable process (bridge) starting at $x$ and conditioned to be at $y$ at time $t$.  For more detail about formula (\ref{Fey.Kac.f}), see \cite{Ba.Sel,Barry}.

Our goal now is to derive a formula for  $\trace$ for the fractional Laplacian similar to the one  in \cite{Ba.Sab} for the Laplacian. 
    
\begin{prop}\label{Fourier-trace-formula}
Let $V\in\calS(\Rd)$, then
\begin{equation*}
\trace= \frac{1}{(2\pi)^d}\ioRd \left(\wh{\pH}(\xi,-\xi)-\wh{\palp}(\xi,-\xi)\right)d\xi.
\end{equation*}
\end{prop}

\begin{proof}
For all $t>0$ and $x,y\in\Rd$, we have 
\begin{align}\label{palp.equat}
\partial_{t}\palp(x,y)&=-(-\Delta)_{x}^{\frac{\alpha}{2}}\palp(x,y) \\
 p_{0}^{(\alpha)}(x,y)&= \delta(x-y) \nonumber
 \end{align}
and
\begin{align}\label{pH.equat}
\partial_{t}\pH(x,y)&=-[(-\Delta)_{x}^{\frac{\alpha}{2}} + V(x)]\pH(x,y) \\
  p_{0}^{H_{V}}(x,y)&= \delta(x-y). \nonumber
\end{align}
 By taking Fourier transform on $\R^{2d}$, we deduce that
\begin{align}\label{Four.palp.equat}
(\partial_{t}+|\xi|^{\alpha})\wh{\palp}(\xi,\eta)&=0 \nonumber\\
  \wh{p_{0}^{(\alpha)}}(\xi,\eta)&=(2\pi)^d \delta(\xi +\eta)
\end{align}
and that 
\begin{align}\label{Four.pH.equat}
(\partial_{t}+ |\xi|^{\alpha})\wh{\pH}(\xi,\eta)&= -\frac{1}{(2\pi)^d}\ioRd \wh{V}(\theta)\wh{\pH}(\xi-\theta, \eta)d\theta \nonumber \\
\wh{p_0^{H_V}}(\xi,\eta)&=(2\pi)^d \delta(\eta+\xi).
\end{align}
Now, by directly solving (\ref{Four.palp.equat}) and (\ref{Four.pH.equat}) we find that
\[
\wh{\palp}(\xi,\eta)=(2\pi)^d\delta(\eta+\xi)e^{-t|\xi|^{\alpha}}
\]
and that 
\begin{equation}\label{import.F.p.equat}
\wh{\pH}(\xi,\eta)-\wh{\palp}(\xi,\eta)= -\frac{1}{(2\pi)^d}\int_{0}^{t}\ioRd e^{-(t-s)|\xi |^{\alpha}}\wh{V}(\theta)\wh{p_s^{H_V}}(\xi-\theta,  \eta)d\theta ds.
\end{equation}
On the other hand, from (\ref{palp.equat}), (\ref{pH.equat}) and Duhamel's Principle we see that 
\[
\pH(x,y)-\palp(x,y)= -\int_0^{t}\ioRd p_{t-s}^{(\alpha)}(x,z) p_{s}^{H_{V}}(z,y)V(z)dzds.
\]
Hence
\begin{eqnarray}\label{trace.2}
\trace&=&\ioRd\left(\pH(x,x)-\palp(x,x)\right)dx \nonumber \\
&=&-\int_0^{t}\iogRd{2}p_{t-s}^{(\alpha)}(x,z)p_{s}^{H_{V}}(z,x)V(z)dzdxds. 
\end{eqnarray}

Expressing the right hand side of (\ref{trace.2}) in terms of Fourier transform we arrive at 
\[
\trace= -\frac{1}{(2\pi)^{2d}}\int_0^{t}\iogRd{3}\wh{V}(\theta)\wh{p_{s}^{H_{V}}}(\mu,\tau)\wh{p_{t- s}^{(\alpha)}}(-\mu -\theta,-\tau)d\mu d\tau d\theta ds. 
\]
Since
$$\wh{p_{t-s}^{(\alpha)}}(-\mu-\theta,-\tau)=(2\pi)^d \delta(\tau+\theta+\mu)e^{-(t-s)|\theta+\mu|^{\alpha}},$$
 we see that 
\[
\trace= -\frac{1}{(2\pi)^{2d}} \int_0^{t}\iogRd{2}e^{-(t-s)|\tau|^{\alpha}}\wh{p_{s}^{H_{V}}}(-\tau-\theta,\tau)\wh{V}(\theta)d\tau d\theta ds.
\]

The conclusion of the proposition  follows by setting $\xi=-\tau$ in last equation,  $\eta=\xi$ in (\ref{import.F.p.equat}) and integrating with respect to $\xi$.
\end{proof}

If we now iterate the equation (\ref{import.F.p.equat}) $J$-times, we obtain 

\begin{cor}  
Let $V\in\calS(\Rd)$, $0<\alpha<2$ and set 
\[
\Fjalp{j}{s}=e^{-(t-s_1)|\xi|^{\alpha}-\mysum{k}{1}{j-1}(s_k - s_{k+1}) |\xi - \mysum{i}{1}{k}\theta_i|^{\alpha}},
\]
where $s=(s_1,...,s_j)$, $s_k <s_{k+1}$ and $\theta=(\theta_1,..., \theta_{j-1})$.  Then for $J\geq2$,
\begin{align*}
&\wh{\pH}(\xi,\eta)-\wh{\palp}(\xi,\eta)= -\frac{1}{(2\pi)^d}\int_0^t \ioRd e^{-(t-s_{1})|\xi |^{\alpha}} \wh{V}(\theta_1)\wh{p_{s_1}^{(\alpha)}}(\xi- \theta_1,\eta)d\theta_1 d s_1 +              \nonumber  \\
&\mysum{j}{2}{J}\frac{(-1)^{j}}{(2\pi)^{jd}}\int_0^t \int_0^{s_1}\cdots \int_0^{s_{j-1}}\iogRd{j} \Fjalp{j}{s}\wh{ p_{s_j}^{(\alpha)}}( \xi -\mysum{i}{1}{j}\theta_i,\eta )\pV{j} d\theta_i ds_i  + \nonumber   \\
&\frac{(-1)^{J+1}}{(2\pi)^{(J+1)d}}\int_0^t \int_0^{s_1}\cdots \int_0^{s_J}\iogRd{(J+1)}\Fjalp{J+1}{s}\wh{p_{s_{J+1}}^{H_V}}(\xi -\mysum{i}{1}{J+1}
\theta_i,\eta )\pV{J+1}d\theta_i ds_i.
\end{align*}
Furthermore, we conclude
\begin{align}\label{trace.3}
&{\trace} =-t\palp(0)\wh{V}(0) + \\ &\mysum{j}{2}{J}\frac{(-t)^{j}}{(2\pi)^{jd}}\lgint{j-1}{j}\Fjalp{j}{t\lambda}e^{-t\lambda_j|\xi|^{\alpha}}\opV d\theta_i d\lambda_i d\lambda_j d\xi  + \nonumber
 \\ &\frac{(-t)^{J+1}}{(2\pi)^{(J+2)d}}\lgint{J}{(J+2)}\Fjalp{J+1}{t\lambda}
\wh{p_{t\lambda_{J+1}}^{H_V}}(\xi -\mysum{i}{1}{ J+1}\theta_i,-\xi )\pV{J+1}d\theta_i d\lambda_{i} d\xi. \nonumber
\end{align}
\end{cor}

\section{Boundedness  of the (J+1)--th term.}\label{sec:boundedcoeff.rem}

Our goal in this section is to provide an upper bound for the absolute value of last expression in \eqref{trace.3} in terms of $\palp(0)$. The following Lemma is a consequence of \eqref{Plancherel}, the inversion formula  and induction.   Therefore its proof is omitted.  
\begin{lem}
Let $J\geq1$ and $\set{p_i}_{i=0}^J \subset \calS(\R^d)$ radial functions. Set $\Psi_{J}(\xi)=\wh{p_{0}}(\xi)\myprod{j}{1}{J}\wh{p_j}(\gamma_j-\xi)$, where $\gamma_j \in \R^d$ are constant vectors. Then,
\begin{equation*}
\wh{\Psi_{J}}(\gamma)= (2\pi)^d\iogRd{J}e^{-i\mysum{j}{1}{J}\langle\gamma_j,x_j\rangle}p_0(\mysum{j}{1}{J}x_j-\gamma)
\myprod{j}{1}{J}p_{j}(x_j)dx_j.
\end{equation*}
\end{lem}

\begin{rmk} If we set $p_{0}=p_{t(1-\lambda_1)}^{(\alpha)}$ and $p_j= p_{t(\incmt{j}{j+1})}^{(\alpha)}$ in the last Lemma, with $\gamma_j =\mysum{k}{1}{j}\theta_k$, we obtain
\begin{equation}\label{FourierF}
\frac{\wh{ F_{J+1}^{(\alpha)}}(t\lambda,x-y,\theta)}{(2\pi)^d}=\iogRd{J}e^{-i\mysum{j}{1}{J}\langle\gamma_j,x_j\rangle}
p_{t(1-\lambda_1)}^{(\alpha)}(\mysum{j}{1}{J}x_j-(x-y))
\myprod{j}{1}{J}p_{t(\lambda_j-\lambda_{j+1})}^{(\alpha)}(x_j)dx_j.
\end{equation}
\end{rmk}
Next, it is known that the transition density $\palp(x,y)$ satisfies the Chapman-Kolmogorov equation, namely,
\begin{equation}\label{Chapman}
\ioRd p_s^{(\alpha)}(a-z)p_{t}^{(\alpha)}(z)dz=p_{t+s}^{(\alpha)}(a),
\end{equation} 
for all $a\in \R^d$ and $t,s>0$. With this equality at hand,  it easily follows that
\begin{equation}\label{iterated1}
\iogRd{J}p_{t(1-\lambda_1)}^{(\alpha)}(\mysum{j}{1}{J}x_j-(x-y))
\myprod{j}{1}{J}p_{t(\lambda_j-\lambda_{j+1})}^{(\alpha)}(x_j)dx_j= p_{t(1-\lambda_{J+1})}^{(\alpha)}(x-y).
\end{equation}

It can also be proved by means of the  inversion formula that
\begin{equation}\label{iterated2}
\iogRd{(J+1)}e^{-i\set{ \langle x,\gamma_{J+1}\rangle+\mysum{j}{1}{J}\langle\gamma_j,x_j\rangle} }
\myprod{i}{1}{J+1}\wh{V}(\theta_i)d\theta_i=(2\pi)^{(J+1)d}V(x)\myprod{k}{1}{J}V(-\mysum{j}{k}{J}x_j+x).
\end{equation}

\begin{prop}\label{last.remainder}
Assume $0<t<1$ and define
\begin{align*}
r_{J+1}(t)=\lgint{J}{(J+2)}&\Fjalp{J+1}{t\lambda}\ \wh{ p_{t\lambda_{J+1}}^{H_V}}(\xi-\mysum{i}{1}{J+1}\theta_i,-\xi)\pV{J+1}d\xi d\theta_i d\lambda_i.
\end{align*}
There exists a positive constant $C=C_{J+1,d,\alpha}(V)$ such that
\begin{equation*}
\abs{r_{J+1}(t)}\leq C\palp(0). 
\end{equation*}

\begin{proof} Set $\gamma_j= \mysum{i}{1}{j}\theta_i$ and 
\begin{equation*}
p(x-y,\set{x_j}_{j=1}^{J})=p_{t(1-\lambda_1)}^{(\alpha)}(\mysum{j}{1}{J}x_j-(x-y))
\myprod{j}{1}{J}p_{t(\lambda_j-\lambda_{j+1})}^{(\alpha)}(x_j).
\end{equation*}
 Then,
by definition of Fourier transform and \eqref{FourierF}, we have that
\begin{align*}
I_{J+1}(t\lambda,\theta)=&\ioRd \Fjalp{J+1}{t\lambda} \wh{ p_{t\lambda_{J+1}}^{H_V}}(\xi-\gamma_{J+1},-\xi)d\xi \nonumber \\
=&\int_{\R^{d}}\Fjalp{J+1}{t\lambda} \int_{\R^{2d}}e^{-i\set{ \langle \xi,x-y \rangle +\langle x,-\gamma_{J+1}\rangle}} p_{t\lambda_{J+1}}^{H_V}(x,y)dxdyd\xi  \nonumber \\
=&\int_{\R^{2d}} p_{t\lambda_{J+1}}^{H_V}(x,y)e^{-i\set{ \langle x,-\gamma_{J+1} \rangle }}
\wh{F_{J+1}^{(\alpha)}}(t\lambda,x-y,\theta)dxdy  \nonumber \\
=&(2\pi)^{d}\iogRd{J}\iogRd{2}e^{-i\set{ \langle x,-\gamma_{J+1} \rangle + \mysum{j}{1}{J}\langle x_j,\gamma_{j}\rangle}} 
 p_{t\lambda_{J+1}}^{H_V}(x,y)p(x-y,\set{x_j}_{j=1}^{J})dxdydx_j.
\end{align*}
Now, because of \eqref{iterated2}, \eqref{iterated1} and  the fact that
$p_{ t\lambda_{J+1} }^{H_V}(x,y)\leq e^{t||V||_{L^{\infty}(\R^d)}}p_{t\lambda_{J+1}}^{(\alpha)}(x,y)$ (which follows from \eqref{Fey.Kac.f}), we obtain from the last equality that 
\begin{align*}
&\abs{\iogRd{(J+1)}I_{J+1}(t\lambda,\theta)\pV{J+1}d\theta_i}= \nonumber \\
&\abs{(2\pi)^{(J+2)d}\iogRd{(J+2)}V(-x)\myprod{i}{1}{J}V(-\mysum{j}{1}{J}x_j-x)
 p_{t\lambda_{J+1}}^{H_V}(x,y)p(x-y,\set{x_j}_{j=1}^{J})dxdydx_j } \leq \nonumber \\
&(2\pi)^{(J+2)d}||V||_{L^{\infty}(\R^d)}^{J}e^{t||V||_{L^{\infty}(\R^d)}}\int_{\R^d}
\abs{V(-x)}\int_{\R^d}p_{t(1-\lambda_{J+1})}^
{(\alpha)}(x-y)p_{t\lambda_{J+1}}^{(\alpha)}(x-y)dydx.
 \end{align*}
 It follows from \eqref{Chapman} that
$$\abs{r_{J+1}(t)}\leq 
\frac{(2\pi)^{(J+2)d}\palp(0)}{(J+1)!}||V||_{L^{\infty}(\R^d)}^{J}e^{t||V||_{L^{\infty}(\R^d)}} ||V||_{L^{1}(\R^d)}.$$
\end{proof}
\end{prop}
As a consequence, we have that 
\begin{equation*}
R_{J+1}(t)=\frac{r_{J+1}(t)}{(2\pi)^{(J+2)d}\palp(0)}
\end{equation*}
is bounded for $0<t<1$.

\section{Heat trace computation by means of subordinators.}\label{sec:trace-sub}
 
In this section we further investigate formula (\ref{trace.3}) involving $\trace$. We start by integrating the function  
\begin{equation}\label{Fj.intergrand}
\Fjalp{j}{t\lambda}e^{ -t\lambda_j |\xi|^{\alpha} }= e^{ -t(1- [\lambda_1-\lambda_j] ) |\xi |^{\alpha}-t\mysum{k}{1}{j-1}( \lambda_k - \lambda_{k+1})|\xi - \mysum{i}{1}{k}\theta_i |^{\alpha}}
\end{equation}
with respect to $\xi$ over $\Rd$.  
The integral could be easily computed when $\alpha=2$ using two elementary facts.  Namely,  for any $\gamma\in \Rd$, 
\begin{equation}\label{basic.facts1}
|\xi+\gamma|^{2}=|\xi|^2+2<\xi,\gamma> +|\gamma|^2 
\end{equation}
and 
\begin{equation}
\ioRd e^{-t|\xi-\gamma|^2}d\xi=\pi^{d/2}t^{-d/2}.
\end{equation}
Unfortunately, we cannot calculate (\ref{Fj.intergrand}) in the same way because there is not a close form for $|\xi+\gamma|^{\alpha}$, when $0<\alpha<2$. Instead, we will follow a probabilistic approach by means of $\alpha/2$--subordinators and their Laplace transform given in \S\ref{Lap.Subor} that relates $|\cdot |^{\alpha}$ to {$|\cdot|^{2}$ to find the value of the integral involving the quantity in \eqref{Fj.intergrand}.
  We begin by observing  that (\ref{CharF-LapS}) implies that for all $c>0$ and $t>0$, 
\begin{equation}\label{Simility.Sub}
e^{-tc |\xi|^{\alpha}}=\texpt{ e^{- t^{\frac{2}{\alpha}}S_{c}|\xi|^2}}.
\end{equation} 
In addition,  for any sequence of numbers $\set{\lambda_k}_{k=1}^{j}$, $j\geq2$, satisfying
\begin{equation}\label{lamb.ineq}
0<\lambda_j<\lambda_{j-1}<...<\lambda_2<\lambda_1 <1, 
\end{equation}
we have
\[S_{1}= S_{ ( 1-\set{ \incmt{1}{j} } )+\incmt{1}{j} }- S_{ \incmt{1}{j} } + \mysum{k}{1}{j-1}\left(S_{\incmt{k}{k+1} +(\incmt{k+1}{j})}-S_{\incmt{k+1}{j}}\right).\]
For  $1 \leq k \leq j-1$ consider the random variables
\begin{equation*}
\spS{ \incmt{k}{k+1}}= S_{ \incmt{k}{k+1} +(\incmt{k+1}{j})}- S_{\incmt{k+1}{j}} 
\end{equation*}
and 
\begin{equation*}
\spS{ 1-(\incmt{1}{j}) }= S_{ 1-\set{\incmt{1}{j}} +\incmt{1}{j}}-S_{\incmt{1}{j}}.
\end{equation*}
Since the process $S$ has independent and stationary increments, we see that 
the random variables $\set{\spS{\incmt{k}{k+1}}, \spS{1- ( \incmt{1}{j})}}_{k=1}^{j-1}$ are independent and furthermore, 

\begin{align} \label{Spec.Sub.eid.Sub}
\spS{ \incmt{k}{k+1} }&\eid S_{\incmt{k}{k+1}}  \nonumber \\
\spS{ 1- (\incmt{1}{j})}&\eid S_{1-(\incmt{1}{j})}.
\end{align}
 We also have, of course, that 
\begin{equation}\label{gen,expr.S1}
\spS {1-(\incmt{1}{j})} + \mysum{k}{1}{j-1}\spS{\incmt{k}{k+1}}=S_{1}. 
\end{equation}

As before let us denote, for simplicity, $\gamma_k=\mysum{i}{1}{k}\theta_i$. It follows from (\ref{Simility.Sub}), (\ref{Spec.Sub.eid.Sub}), (\ref{gen,expr.S1}) and the independence of $\set{\spS{\incmt{k}{k+1}}, \spS{1- ( \incmt{1}{j})}}_{k=1}^{j-1}$ that 

\begin{align}\label{trick.expt}
&e^{ -t( 1- [\incmt{1}{j}] )|\xi |^{\alpha} } \myprod{k}{1}{j-1}e^{ -t (\incmt{k}{k+1}) |\xi - \gamma_k |^{\alpha} }= \nonumber \\
&\texpt{ \exp\left( -t^{2/\alpha}\spS{ 1 -( \incmt{1}{j} ) }|\xi|^2 \right)}\texpt{ \exp\left( -t^{2/\alpha} \mysum{k}{1}{j-1}\spS{ \incmt{k}{k+1} }|\xi - \gamma_k|^2 \right) }= \nonumber \\
&\texpt{ \exp\left(-t^{2/\alpha} \set{  \spS{1 -(\incmt{1}{j}) } |\xi|^2 + \mysum{k}{1}{j-1}\spS{ \incmt{k}{k+1} }|\xi - \gamma_k|^2 } \right)}.
\end{align}

Next, consider the random variable
\begin{equation}\label{Lj def}
\Lj= \mysum{k}{1}{j-1}\spS{\incmt{k}{k+1}}|\gamma_k|^2 -\frac{1}{S_{1}}\left|\mysum{k}{1}{j-1}\spS{\incmt{k}{k+1}}\gamma_k\right|^2, 
\end{equation}
where $\lambda= (\lambda_1,...,\lambda_j)$ satisfies (\ref{lamb.ineq}).  By (\ref{basic.facts1}), (\ref{gen,expr.S1}}) and completing squares we easily get that 
  
\begin{align*}
\spS{1 -( \incmt{1}{j}) } |\xi |^2 +\mysum{k}{1}{j-1}\spS{\incmt{k}{k+1}}|\xi - \gamma_k|^2 = S_{1} \abs{\xi - \frac{1}{S_{1}}\mysum{k}{1}{j-1}\spS{\incmt{k}{k+1}}\gamma_k}^2 + \Lj .
\end{align*}
Also,  observe that by (\ref{basic.facts1}) and the scaling property (\ref{Scal.prop}) we have 
\begin{eqnarray} \label{prev.int.Fj.val}
\ioRd \exp\left( -t^{2/\alpha}S_{1}\abs{ \xi - \frac{1}{S_{1}}\mysum{k}{1}{j-1}\spS{\incmt{k}{k+1}}\gamma_k }^2\right)d\xi&=&\pi^{d/2}t^{-d/\alpha}S_{1}^{-d/2} \nonumber \\
&=&\const\palp(0)S_{1}^{-d/2}, 
\end{eqnarray}
over the the set where $0<S_{1}<\infty$.
Here,  
\begin{equation}\label{imp.const}
\const= \frac{\pi^{d/2}}{ p_{1}^{(\alpha)}(0)}. 
\end{equation} 
 
We now combine these calculations to find the value of the desired  integral.  More precisely, we have

\begin{lem}
Let $ \lambda= (\lambda_1,...,\lambda_j)$ satisfy (\ref{lamb.ineq}), $\gamma_k=\mysum{i}{1}{k}\theta_i$ and $\theta= (\theta_1,...,\theta_{j-1})$.  Then, 
\[ \Lj\geq0, \,\,\,\, \textrm{a.s.}\]
and 
 \begin{align}\label{Fj.int,val}
\ioRd \Fjalp{j}{t\lambda}e^{-t\lambda_j|\xi|^{\alpha}}d\xi &=\const\palp(0)\texpt{S_{1,\frac{\alpha}{2}}^{-d/2}e^{-t^{2/\alpha}\Lj}}.
\end{align}
\end{lem}

\begin{proof}
Assume $\gamma_k=(b_{1,k},...,b_{d,k})$.  By Cauchy-Schwarz inequality and 
\eqref {gen,expr.S1}

\begin{eqnarray}
\abs{ \mysum{k}{1}{j-1}\spS{\incmt{k}{k+1}}\gamma_k}^2 &=& \mysum{m}{1}{d}\left\{ \mysum{k}{1}{j-1}\spS{\incmt{k}{k+1} }b_{m,k} \right\}^2=\mysum{m}{1}{d}\left\{ \mysum{k}{1}{j-1}\set{\spS{\incmt{k}{k+1} }}^{\frac{1}{2}}b_{m,k}\set{\spS{\incmt{k}{k+1} }}^{\frac{1}{2}} \right\}^2 \nonumber \\
&\leq & \mysum{m}{1}{d}\left\{ \mysum{k}{1}{j-1}\spS{\incmt{k}{k+1}} b_{m,k}^2 \right\}\mysum{k}{1}{j-1}\spS {\incmt{k}{k+1}} 
\leq S_{1}\mysum{k}{1}{j-1}\spS{\incmt{k}{k+1}}|\gamma_k|^2.  \nonumber
\end{eqnarray}

In fact, under the convention that 
$\mysum{r}{1}{0}=0$, we have for  $j\geq2$, 
\begin{equation}\label{exp.form.Lj}
\Lj=S_{1}^{-1}\left[ \spS{1 - (\incmt{1}{j})} \mysum{k}{1}{j-1}\spS{\incmt{k}{k+1},}|\gamma_k|^2 +
\mysum{r}{1}{j-2}\mysum{s}{r+1}{j-1}\spS{\incmt{r}{r+1}}\spS{\incmt{s}{s+1}}|\gamma_r-\gamma_s|^2\right].
\end{equation}
On the other hand, (\ref{Fj.int,val}) follows by integrating (\ref{trick.expt}) with respect to $\xi$, applying  Fubini's Theorem to (\ref{trick.expt}) and using (\ref{prev.int.Fj.val}).
\end{proof}

\begin{rmk}
We note that from (\ref{exptofS_1}) and the fact that $0<S_1<\infty$, a.s, 
\begin{equation}\label{expt,S1,Lj}
0<\texpt{ S_{1,\alpha/2}^{-d/2}e^{-t^{2/\alpha}\Lj }} \leq \texpt{S_{1,\alpha/2}^{-d/2}}<\infty.
\end{equation}
In fact, all the above results are true for $\alpha=2$ in which case $S_{1,1}=1$ and all our calculation considerably simplify. 
\end{rmk}

\section{Bounds for remainders and coefficients.} \label{sec:coeffic.sect}

We observe that the exponential function is involved in (\ref{Fj.int,val}) and that this term is part of the expression for $\trace$ in (\ref{trace.3}). Our next step is to use a Taylor expansion of the exponential function with a particular remainder to obtain a finer estimate for the trace. This implies, as the reader may note, that in (\ref{Fj.int,val}) we will have to deal with expectations. Hence our goal in this section is to give conditions to guarantee the finiteness of these expectations. Once this is done, it will follow easily that the coefficients and remainders to appear in (\ref{trace.3}) are also finite and bounded, respectively.

We recall the well known expansion for the exponential function 
\begin{equation}\label{Tay.exp}
e^{-x}= \mysum{n}{0}{m-1}\frac{(-1)^n}{n!}x^{n} + \frac{(-1)^m}{m!}x^{m}e^{-x \beta_{m}(x)}
\end{equation}
valid for every $x\geq0$ and $m\geq1$, where we call $\beta_{m}(x)\in (0,1)$ the remainder of order m. With this expansion at hand, we now introduce two functions from which  we will obtain the desired finer estimates in the trace formula  (\ref{trace.3}). For $j\geq2$,
  
\begin{align}\label{gen.term}
&T_{d}(j,t)=\lgint{j-1}{j} \Fjalp{j}{t\lambda}e^{-t\lambda_j|\xi|^{\alpha}} \opV d\theta_i d\lambda_i d\xi d\lambda_j= \\ &\const\palp(0)\lgint{j-1}{(j-1)}\texpt{ S_{1,\alpha/2}^{-d/2} e^{-t^{\frac{2}{\alpha}} \Lj}}
 \opV d\theta_i d \lambda_i d\lambda_j. \nonumber 
\end{align}
The remainder function is 

\begin{align}\label{remainder.term}
R_{j,d}^{(m)}(t)=\frac{\const}{m!(2\pi)^{jd}}\lgint{j-1}{(j-1)}&\texpt{ S_{1,\alpha/2}^{-d/2}\set{\Lj}^{m} e^{- \beta_{m,j}^{*}(t) }} \nonumber \\ \times&\opV d\theta_i d \lambda_i d\lambda_j,
\end{align}
where the random functions $ \beta_{m,j}^{*}(t)=t^{2/\alpha}\Lj \beta_{m} (-t^{ 2/\alpha }\Lj)$ are nonnegative. 

\begin{rmk}
We note by (\ref{expt,S1,Lj}) and (\ref{gen.term}) that  

\begin{equation*} \label{finite.T}
\abs{ \frac{ T_{d}(j,t) } { (2\pi)^{jd} \palp(0) }} \leq \frac{ \const \texpt{S_{1,\alpha/2}^{-d/2}} }{ j!(2\pi)^{jd} }
\iogRd{(j-1)}\abs{\opV}d\theta_i,
\end{equation*}  
for all $j\geq2$, $d\geq 1$. Observe that the left hand side is finite since   $V\in\calS(\Rd)$, proving at the same time the finiteness of $T_{d}(j,t)$, for all $t>0$.
\end{rmk}

We now proceed to prove the finiteness of the remainder-functions in (\ref{remainder.term}) and define the coefficients given in Theorem \ref{mainthm}.

\begin{lem}\label{lem.fin.rem}
Assume $M\geq1$ is an integer satisfying $M < \frac{\alpha+d}{2}$. Then, for all $t\geq0$ and $j\geq2$,

\begin{equation*}\label{finite.remainder}
\abs{ R_{j,d}^{(M)}(t) } \leq C_{j,d,M}\iogRd{(j-1)}\left(\mysum{k}{1}{j-1}|\gamma_{k}|^2\right)^{M}
\abs {\opV} d\theta_i,
\end{equation*}
where $$C_{j,d,M}=\frac{\const\texpt{ S_{1,\alpha/2}^{M-d/2} }}{j!M!(2\pi)^{jd}},\,\,\,\, \gamma_k=\sum\limits_{i=1}^{k}\theta_i.$$ Furthermore, 
\begin{enumerate}
\item[(a)] for all integers $n\leq M$, $$0 \leq \texpt{ S_{1,\frac{\alpha}{2}}^{-d/2}\left\{ \Lj \right\}^n}< \infty,$$
\item[(b)] and for all $j\geq2$,
\begin{eqnarray*}\label{expansion.T}
\frac{T_{d}(j,t)}{ (2\pi)^{jd} \palp(0)} =\mysum{n}{0}{M-1}(-1)^{n}C_{n,j}^{(\alpha)}(V)t^{\frac{2n}{\alpha}} + (-1)^Mt^{\frac{2M}{\alpha}}R_{j,d}^{(M)}(t),
\end{eqnarray*}
where 
\begin{align*}\label{definition.C(V)}
C_{n,j}^{(\alpha)}(V)=\frac{\const}{ (2\pi)^{jd} n!}&\lgint{j-1}{(j-1)} \texpt{ S_{1,\alpha/2}^{-d/2}\left\{\Lj\right\}^{n}} 
\opV d\theta_i d\lambda_i d\lambda_j.
\end{align*}
\end{enumerate}
\end{lem}

\begin{proof}
We start by observing that the condition $M < \frac{d+\alpha}{2}$ guarantees that $0\leq \texpt{S_{1,\frac{\alpha}{2}}^{n-d/2}}<\infty$ for all integers $n\leq M$, according to (\ref{gen.expt.S1}).  From this we proceed to prove (a) as follows. 
Recall that
\[
\spS{1-(\incmt{1}{j})} + \mysum{k}{1}{j-1}\spS{\incmt{k}{k+1}}=S_{1,\alpha/2}.
\]

  It follows  from (\ref{Lj def}) and last equality that 
\begin{equation*}
0<\Lj\leq\mysum{k}{1}{j-1}\spS{\incmt{k}{k+1}}|\gamma_k|^2\leq S_{1,\alpha/2}  \mysum{k}{1}{j-1}|\gamma_{k}|^{2}
\end{equation*}
Then, from the last inequality we conclude that  for all integer $n\leq M$,  
\begin{equation*}
\texpt{S_{1,\alpha/2}^{-d/2}\set{ \Lj}^{n}}\leq\texpt{ S_{1,\alpha/2}^{n-d/2}}\left(\mysum{k}{1}{j-1}|\gamma_{k}|^{2}\right)^n.
\end{equation*}
Thus (a) now follows easily from  last inequality. On the other hand,  (b) follows from the Taylor expansion (\ref{Tay.exp}) applied to (\ref{gen.term}). We remark that 
$C_{n,j}^{(\alpha)}(V)= R_{j,d}^{(n)}(0)$. Therefore the last expression is finite, according to (\ref{finite.remainder}).
\end{proof}

\begin{rmk}
Lemma \ref{lem.fin.rem} allows us to bound the remainders by a constant for all $t\geq0$ and shows the finiteness of both the remainders and the coefficients $C_{n,j}^{(\alpha)}(V)$ by proving the finiteness of the expectations under the condition $n\leq M<\frac{d+\alpha}{2}$. Indeed, this condition is introduced to make sense of the Taylor expansion of order M when it is applied to the function (\ref{expt,S1,Lj}). As a consequence our results are dimensional dependent. The reason why this does not happen when  $\alpha=2$ is that in this case the time change is trivial, $S_{t,1}=t$, and $L_j^{(2)}$ is nonrandom function.  These two facts considerably reduce all above computations, thereby the dimension  only appearing in the integrals involving $\wh{V}$, which are finite since $V\in \calS(\Rd)$.
\end{rmk}

\section{An improvement for dimension $d=1, 2, 3$.}\label{sec:improvement}
We recall the basic inequality
$$\left( \mysum{k}{1}{j-1}a_k\right)^M\leq (j-1)^{M-1}\mysum{k}{1}{j-1}a_k^M,$$
valid for all $j\geq 2$ and positive numbers $\set{a_k}_{k=1}^{j-1}.$

From last inequality and \eqref{exp.form.Lj}, it follows for all integer $M\geq1$ that
there exists a constant $C_{j,M}>0$ such that
\begin{align}\label{case.small.d}
C_{j,M}^{-1}\texpt{S_{1}^{-d/2}\set{\Lj}^M}\leq \mysum{k}{1}{j-1}\texpt{ 
 \left(\spS{1 - (\incmt{1}{j})}\spS{\incmt{k}{k+1}}\right)^M S_1^{-M-\frac{d}{2}}}
 |\gamma_k|^{2M} + \nonumber \\
\mysum{r}{1}{j-2}\mysum{s}{r+1}{j-1}\texpt{(\spS{\incmt{r}{r+1}}\spS{\incmt{s}{s+1}})^M S_1^{-M-\frac{d}{2}}}|\gamma_r-\gamma_s|^{2M},
\end{align}
whenever the expectations involved in the last expression are finite.
The purpose of this section is to provide conditions under which these last expectations are finite for dimension $d =1,2$ and $3$.

We proved in \S\ref{sec:trace-sub} that
\begin{equation*}
\spS {1-(\incmt{1}{j})} + \mysum{k}{1}{j-1}\spS{\incmt{k}{k+1}}=S_{1},
\end{equation*}
for $\set{\lambda_k}_{k=1}^{J}$ satisfying \eqref{lamb.ineq} and where the random variables on the right hand side of the last equality are independent. In particular, it follows that
 $S_{1}\geq \spS{l_0}+\spS{l_1}$ for any distinct $l_0,l_1 \in \set{1-(\incmt{1}{j}),\incmt{k}{k+1}}_{k=1}^{j-1}$. Now, observe that each expectation in \eqref{case.small.d} can be written as 
 $\texpt{(\spS{l_1}\spS{l_0})^M S_1^{-M-\frac{d}{2}}}$, and these expectations satisfy
\begin{equation*}
\texpt{\frac{(\spS{l_1}\spS{l_0})^M}{ S_1^{M+\frac{d}{2}}}}\leq \texpt{\frac{(\spS{l_1}\spS{l_0})^M}
{(\spS{l_1}+\spS{l_0})^{M+\frac{d}{2}}}}.
\end{equation*}

\begin{lem}\label{finer.inequality.Lem}
Let $j\geq2$ and $\set{\lambda_k}_{k=1}^{j}$ satisfying \eqref{lamb.ineq}. Let $l_0,l_1$ be two distinct numbers in $\set{1-(\incmt{1}{j}),\incmt{k}{k+1}}_{k=1}^{j-1}$. Then,
\begin{equation*}\label{Lemma1smallcase}
0\leq\texpt{\frac{(\spS{l_0}\spS{l_1})^M}{(\spS{l_0}+\spS{l_1})^{M+d/2}}}<\infty,
\end{equation*}
provided that $M/2-d/4<\alpha/2$.

In particular, when
\begin{enumerate}
\item[i)] M=1, if $\frac{2-d}{2}< \alpha$.
\item[ii)] M=2, if $\frac{4-d}{2}< \alpha$.
\end{enumerate}
\end{lem}
\begin{proof}
Because of the inequality $2(ab)^{1/2}\leq a + b$, for any $a,b\geq0$, we have that
\begin{align*}
\texpt{\frac{(\spS{l_0}\spS{l_1})^M}{(\spS{l_0}+\spS{l_1})^{M+d/2}}}\leq
2^{-M-d/2}\texpt{\frac{(\spS{l_0}\spS{l_1})^M}{(\spS{l_0}\spS{l_1})^{M/2+d/4}}}.
\end{align*}
Now, recall that $\spS{l_0}$ and $\spS{l_1}$ are independent and $\spS{l_i} \eid  l_i^{2/\alpha}S_{1,\alpha/2}$ . Therefore,
\begin{equation}\label{finer.inequaality}
\texpt{\frac{(\spS{l_0}\spS{l_1})^M}{(\spS{l_0}\spS{l_1})^{M/2+d/4}}}= 
(l_0 l_1)^{\frac{2}{\alpha}\set{\frac{M}{2}-\frac{d}{4}}}
\left(\texpt{S_{1,\frac{\alpha}{2}}^{\frac{M}{2}-\frac{d}{4}}}\right)^{2}.
\end{equation}
 The result follows from the inequality $M/2-d/4<\alpha/2$ which guarantees the finiteness  of the last expectation.
\end{proof}

As an application of Lemma \ref{finer.inequality.Lem}, we have 
\begin{cor}\label{improve.cor}

Assume $j\geq2$.
\begin{enumerate}
\item[(i)] For $\frac{1}{2}<\alpha<2$ and $d=1$, we have  $$\texpt{S_{1,\frac{\alpha}{2}}^{-1/2}\Lj}<\infty.$$
 \item[(ii)]  For $d=1$ and $\frac{3}{2}<\alpha<2$,
 $d=2$ and $1<\alpha<2$, 
 $d=3$ and $\frac{1}{2}<\alpha<2$, we have 
\begin{equation*}
\texpt{S_{1,\frac{\alpha}{2}}^{-d/2}\set{ \Lj}^{2}}<\infty.
\end{equation*}
\end{enumerate}
\end{cor}

The following is  a version of Lemma \ref{lem.fin.rem} for dimension $1,2,$ and $3$ where the condition $M<\frac{d+\alpha}{2}$ is  replaced by $-1< M/2-d/4< \alpha/2$. 

\begin{lem}\label{smallcase}
\hspace{20mm}
\begin{enumerate}
\item[(i)] For d=1, and M=1, we have for all $\frac{1}{2}<\alpha<2$ and $j\geq2$ that
\begin{align*}
\frac{T_{1}(j,t)}{ (2\pi)^{j} \palp(0)} =C_{0,j}^{(\alpha)}(V) - t^{\frac{2}{\alpha}}R_{j,1}^{(1)}(t).
\end{align*}
\item[(ii)] For M=2 and $j\geq2$, we have
\begin{align*}
\frac{T_{d}(j,t)}{ (2\pi)^{dj} \palp(0)} =\mysum{n}{0}{1}(-1)^{n}C_{n,j}^{(\alpha)}(V)t^{\frac{2n}{\alpha}} + t^{\frac{4}{\alpha}}R_{j,d}^{(2)}(t),
\end{align*}
when $d=1$  and $\frac{3}{2}<\alpha<2$, $d=2$ and $1<\alpha<2$, or $d=3$ and 
$\frac{1}{2}< \alpha< 2 $.
\end{enumerate}
where the remainders $\abs{R_{j,d}^{(2)}(t)}$ are bounded by a constant for all $t\geq 0$,
according to Corollary \ref{improve.cor} and the fact that $M/2-d/4>-1$ for all $M,d$ as stated above.
\end{lem}
\begin{proof}
We start by recalling that for any $j\geq2$ 
\begin{equation*}
I_{j}=\set{\lambda=(\lambda_1,...,\lambda_{j}): 0<\lambda_{j}<\lambda_{j-1}<...<\lambda_1<1}.
\end{equation*}

Next, under the notation given in Lemma \ref{lem.fin.rem}, we have that $\abs{C_{1,j}^{(\alpha)}(V)}$ and $\abs{ R_{j,d}^{(M)}(t)}$, $M=1,2$  are bounded by
\begin{align*}
\int_{I_j}\int_{\R^{(j-1)d}}\texpt{S_{1}^{-d/2}\set{\Lj}^M}\abs{\opV}d\lambda_j d\lambda_i d\theta_i.
\end{align*}
This last expression is also bounded, based on the facts given at the beginning of this section, up to some positive constant by terms of the form
\begin{equation*}
\int_{I_j}(l_0l_1)^{\frac{2}{\alpha}(\frac{M}{2}-\frac{d}{4})}d\lambda_j d\lambda_i\left(\texpt{S_{1,\frac{\alpha}{2}}^{\frac{M}{2}-\frac{d}{4}}}\right)^2\int_{\R^{(j-1)d}}
\mysum{k}{1}{j-1}\abs{\gamma_k}^{2M}\abs{\opV}d\theta_i
\end{equation*}
where $l_0=l_0(\lambda)$ and $l_1=l_1(\lambda)$ are two distinct numbers in
$\set{1-(\incmt{1}{j}),\incmt{k}{k+1}}_{k=1}^{j-1}$.

Now the term 
$$\int_{I_j}(l_0l_1)^{\frac{2}{\alpha}(\frac{M}{2}-\frac{d}{4})}d\lambda_j d\lambda_i
=\int_{0}^{1}\int_{0}^{\lambda_1}...\int_{0}^{\lambda_{j-1}}
(l_0l_1)^{\frac{2}{\alpha}(\frac{M}{2}-\frac{d}{4})}d\lambda_j...d\lambda_1$$ is clearly finite when $M/2-d/4\geq0$, which is the case for $M=1,2$, $d=1,2$ and $M=2$, $d=3$. But, when $M=1$ and $d=3$, we obtain $-1<M/2-d/4=1/2-3/4=-1/4$ and this case deserve special attention.

We observe that  for all $1\geq\lambda_i>\lambda_{k+1}$ we have 
$$\int_{0}^{\lambda_i}\lambda_{k+1}(\lambda_i-\lambda_{k+1})^{-\frac{1}{2\alpha}}d\lambda_{k+1}\leq
\int_{0}^{\lambda_i}(\lambda_i-\lambda_{k+1})^{-\frac{1}{2\alpha}}d\lambda_{k+1}\leq \int_{0}^{1}w^{-\frac{1}{2\alpha}}dw=\frac{2\alpha}{2\alpha-1}$$ and
$$\int_{0}^{\lambda_{j-1}}(1-\lambda_1+\lambda_j)^{-\frac{1}{2\alpha}}(\lambda_{j-1}-\lambda_j)^{-\frac{1}{2\alpha}}d\lambda_j\leq(1-\lambda_1)^{-\frac{1}{2\alpha}}\frac{2\alpha}{2\alpha-1}$$
provided that $\alpha>\frac{1}{2}$. Then, it is not difficult to see that
\begin{align*}
\int_{I_j}(l_0l_1)^{-\frac{1}{2\alpha}}d\lambda_j d\lambda_i \leq \left(\frac{2\alpha}{2\alpha-1}\right)^2.
\end{align*}
\end{proof}

\section{Proof of  Theorem \ref{mainthm}.}\label{sec:proof.theorem}
{\bf Proof of part (a):} Recall that, for $J\geq2$, we have defined

\begin{equation*}\label{J+1-remainder}
R_{J+1}(t)= \frac{ r_{J+1}(t)}{(2\pi)^{(J+2)d} \palp(0)}
\end{equation*}
and also showed, according to Proposition \eqref{last.remainder}, that this remainder is   bounded by a constant for  $0\leq t<1$.
Also $M<\frac{d+\alpha}{2}$ implies that $R_{2,d}^{(M)}(t),...,R_{(M+1),d}^{(1)}(t)$ and $R_{M+2}(t)$ are, according to Lemma \ref{finite.remainder}, bounded by a constant for $0<t$.

 Next, $(a)$ follows by substituting the terms found in Lemma \ref{finite.remainder} into \eqref{trace.3}. More precisely, 
\begin{align}\label{general.trace.formula}
\frac{\trace}{\palp(0)}=- t\ioRd V(\theta)d\theta + \mysum{j}{2}{J} \mysum{n}{0}{M-1}(-1)^{n+j}C_{n,j}^{(\alpha)}(V)t^{\frac{2n}{\alpha}+j}\nonumber \\ 
(-t)^{J+1}R_{J+1}(t)+ \mysum{j}{2}{J}
(-1)^{j+M}t^{j  + \frac{2M}{\alpha}}R_{j,d}^{(M)}(t).
\end{align}

Therefore, it suffices to take
\begin{align*}\label{REMAINDER}
R_{J+1}^{(\alpha)}(t)=t^{-\Phi_{J+1}^{(\alpha)}(M)}\left\{(-t)^{J+1}R_{J+1}(t)+ \mysum{j}{2}{J}
(-1)^{j+M}t^{j  + \frac{2M}{\alpha}}R_{j,d}^{(M)}(t)\right\},
\end{align*}
and this proves Theorem \ref{mainthm}.

Moreover, Theorem \ref{2.mainthm} now follows by noticing that
\begin{align*}
\mysum{j}{2}{J} \mysum{n}{0}{M-1}(-1)^{n+j}C_{n,j}^{(\alpha)}(V)
t^{\frac{2n}{\alpha}+j}= \sum\limits_{\substack {\frac{2n}{\alpha}+j<\Phi_{J+1}^{(\alpha)}(M) \\ 2\leq j\leq J,\,\, 0\leq n\leq M-1}}(-1)^{n+j}C_{n,j}^{(\alpha)}(V)t^{\frac{2n}{\alpha}+j} \\+\sum\limits_{\substack {\frac{2n}{\alpha}+j\geq\Phi_{J+1}^{(\alpha)}(M) \\ 2\leq j\leq J,\,\, 0\leq n\leq M-1}}(-1)^{n+j}C_{n,j}^{(\alpha)}(V)t^{\frac{2n}{\alpha}+j}
\end{align*}
and
$$t^{\frac{2n}{\alpha}+j}=\mathcal{O}(t^{\Phi_{J+1}^{(\alpha)}(M)}),$$ as $\tgo$, provided that
$\frac{2n}{\alpha}+j\geq\Phi_{J+1}^{(\alpha)}(M)$.

\bigskip
{\bf Proof of part (b):} We begin by recalling several basic facts about weak convergence.  A convenient reference for this material is \cite{Billingsley}. 
Let  $X_r$, $X$ be $k$-dimensional random vectors, possibly defined on different probability spaces. We recall  that $X_r$ converges weakly to $X$, denoted by $X_r \Rightarrow X$, if  $$\lim\limits_{r\to\infty}F_{X_{r}}(x)=F_{X}(x),$$
for every continuity point $x$ of $F_{X}$, where  $F_{X_r}$ and $F_{X}$ are the distribution functions of $X_r$ and $X$, respectively. The followings statements are consequences of weak convergence and we state them here as a facts, referring the reader again to \cite{Billingsley}.
\begin{enumerate}
\item[F1] Suppose that $h$: $\R^{k}\rightarrow \R^j$ is measurable and that the set $D_h$ of discontinuities of $h$ is measurable. If $X_r \Rightarrow X$ and $P(X\in D_h)=0$,
then $h(X_r)\Rightarrow h(X)$.
\item[F2] Let $X_r, X$ be $k$-dimensional random--vectors. Then, $X_r \Rightarrow X$ if and only if for every bounded, continuous function $f$, we have
$$\lim_{r\to\infty}\texpt{f(X_r)}=\texpt{f(X)}.$$
\item[F3] Let $Y_r, Y$ be real valued random variables. If $Y_r \Rightarrow Y$ and the $\set{Y_r}_{r\in \mathbb{N}}$ is uniformly integrable, then Y is integrable and
$$\lim_{r\to \infty}\texpt{Y_r}=\texpt{Y}.$$ It is a standard fact that if
$$\sup\limits_{r}\texpt{\abs{Y_r}^{p}}<\infty,$$ 
for some $p>1$, then $\set{Y_r}_{r\in \mathbb{N}}$ is uniformly integrable.
\item[F4] Let  $X_r$, $X$ be $k$-dimensional random vectors. Then, $X_r \Rightarrow X$ if and only if for all $v \in \R^k$, $<v,X_r>\,\ \Rightarrow\,\, <v, X>.$
\end{enumerate}

Since we are only interested in $\alpha$'s close to $2$, it suffices to prove that if 
$n,d \geq 1$ are positive integers satisfying $n\leq\frac{1+d}{2}$,  then for all $ j\geq 2$ we have
\begin{equation*} 
\lim_{r\to \infty}C_{n,j}^{(\alpha_r)}(V)=C_{n,j}^{(2)}(V),
\end{equation*}
for any sequence  $\set{\alpha_r}_{r \in \mathbb{N}}$  satisfying
 $\frac{3}{2}< \alpha_r<2$ and  $ \alpha_r \uparrow 2.$
 
To prove last statement, we need to introduce some notation. We recall that 
 $I_j\subset\R^{j}$ has been defined as
\begin{equation*}
I_{j}=\set{\lambda=(\lambda_1,...,\lambda_{j}): 0<\lambda_{j}<\lambda_{j-1}<...<\lambda_1<1}.
\end{equation*}
 We also set 
 \begin{align*}
 X_r(\lambda)&=(S^*_{1-(\incmt{1}{j}), \frac{\alpha_r}{2}}, S^*_{\incmt{j-1}{j},\frac{\alpha_r}{2}},..., S^*_{\incmt{1}{2},\frac{\alpha_r}{2}}), \\ \\
 X(\lambda) &=(1-(\incmt{1}{j}), \incmt{j-1}{j},... , \incmt{1}{2}),\\ \\
h_{n,d}(x_0,x_1,...,x_j)&=\left\{\begin{array}{cc}
       \frac{ \left( x_0\mysum{k}{1}{j-1}x_k \abs{\gamma_k}^2+\mysum{m}{1}{j-2}\mysum{s}{m+1}{j-1}x_m x_s\abs{\gamma_m-\gamma_s}\right)^n}
    {\left( \mysum{k}{0}{j-1}x_k \right)^{n+\frac{d}{2}}}, & \mbox{for $x_0>0$,..., $x_j>0$},\\ \\
       0, & \mbox{otherwise.}
              \end{array}
\right.
\end{align*}
 With this notation, 
 $$h_{n,d}(X_r(\lambda))=S_{1,\frac{\alpha_r}{2}}^{-d/2}\set{L_j^{(\alpha_r)}(\lambda,\theta)}^n,$$ 
 $$h_{n,d}(X(\lambda))=\set{L_{j}^{(2)}(\lambda,\theta)}^n.$$
\\
We now divide our proof into 5 steps.

{\bf Step 1.} $X_r(\lambda) \Rightarrow X(\lambda)$.

 To see this, we recall that for 
$t\in \set{1-(\incmt{1}{j}),\incmt{k}{k+1}}_{k=1}^{j-1}$ and $\lambda>0$, 
$$\texpt{e^{-\lambda \spS{t,\alpha_r/2}}}=e^{-t\lambda^{\alpha_r/2}}.$$
This expectation corresponds to the Laplace transform of $\spS{t,\alpha_r/2}$ and uniquely determines its distribution. We conclude that
$$\lim\limits_{r\rightarrow+\infty}\texpt{e^{-\lambda \spS{t,\alpha_r/2}}}=
e^{-t\lambda}.$$ 
Thus, $\spS{t,\alpha_r/2}\Rightarrow t$. On the other hand, due to the independence of  $\spS{1-(\incmt{1}{j}),\alpha_r/2}$, $\spS{\incmt{j-1}{j},\alpha_r/2}$,\dots \\$\spS{\incmt{1}{2},\alpha_r/2}$, the fact that $f(z)=e^{iz}$ is bounded in $\R$ and F2, we also obtain  for every $v\in \R^{j+1}$ that
\begin{equation*}
\texpt{e^{-i<v,X_r(\lambda)>}}\rightarrow \texpt{e^{-i<v,X(\lambda)>}}.
\end{equation*}
as $r\to \infty$. The result follows from F4.
\bigskip

{\bf Step 2.} $h_{n,d}(X_{r}(\lambda)) \Rightarrow h_{n,d}(X(\lambda))$. 

We note that each component of the vector $X(\lambda)$ is positive. Thus, by our definition of $h_{n,d}$, it is clear that $X(\lambda)$ belongs to the set of continuity points  of $h_{n,d}$. Then,
$P(X(\lambda) \in D_{h_{n,d}})=0$ and the result follows from F1.
\bigskip

{\bf Step 3.} $\set{h_{n,d}(X_r)}_{r \in \mathbb{N}}$ is uniformly integrable.

We shall show that there exist a $p>1$ and a function $C(n,d,p,\theta)>0$ such that

\begin{equation}\label{ineq.unif.int}
\sup\limits_r \texpt{\set{h_{n,d}(X_r(\lambda))}^{p}}\leq C(n,d,p,\theta).
\end{equation}

To do this, we consider two cases to determine a proper $p$.
\\

{\bf Case 1.} Suppose $n-\frac{d}{2}\leq 0$.  In \S \ref{sec:coeffic.sect}, we proved that 
\begin{equation*}\label{unif.int}
h_{n,d}(X_r(\lambda))\leq S_{1,\frac{\alpha_r}{2}}^{n-\frac{d}{2}}\left(\mysum{k}{1}{j-1}\abs{\gamma_k}^2\right)^n. 
\end{equation*}

To prove (\ref{ineq.unif.int}), it suffices to show that
 $\sup\limits_{r}\texpt{S_{1,\frac{\alpha_r}{2}}^{p\left(n-\frac{d}{2}\right)}}$ is bounded for some $p>1$.
Recall from \S \ref{sec:stable-sub} that 
\begin{equation}\label{exptobound}
0<\texpt{S_{1,\frac{\alpha_r}{2}}^{p\left(n-\frac{d}{2}\right)}}=
\frac{\Gamma(1-\frac{2p}{\alpha_r}\left(n-\frac{d}{2}\right))}
{\Gamma(1-p\left(n-\frac{d}{2}\right))}<\infty, 
\end{equation}
 provided that
\begin{equation}\label{pcondition}
p\left(n-\frac{d}{2}\right)<\frac{\alpha_r}{2}<1.
\end{equation}

If $n-\frac{d}{2}=0$, we take $C(n,d,p,\theta)=\left(\mysum{k}{1}{j-1}\abs{\gamma_k}^2\right)^{pn}$ and any $p>1$, since clearly in this case
$\texpt{S_{1,\frac{\alpha_r}{2}}^{p\left(n-\frac{d}{2}\right)}}=1$.

If $n-\frac{d}{2}<0$, it is clear that  \eqref{pcondition} is satisfied for any $p>1$. But, we wish to pick $p$ so that \eqref{exptobound} is uniformly bounded in $r$. To do this in a suitable manner, we require the following well--known property of the the gamma function(see \cite{Pascal}). There exists $\mu_0\in(1,2)$ such that $\Gamma(z)$ is decreasing on $(0,\mu_0]$ and increasing over $(\mu_0,+\infty)$.  
Now, from the last  property and the fact that each $\alpha_r$ satisfies $1<\frac{2}{\alpha_r}<2$, it follows that for any $p>\max\set{1,\frac{1}{\frac{d}{2}-n}}$, 
\begin{equation}
\Gamma(2)\leq\Gamma\left(1+\frac{2p}{\alpha_r}\left(\frac{d}{2}-n\right)\right)\leq\Gamma\left(1 + 2p\left(\frac{d}{2}-n\right)\right).
\end{equation}
Therefore,
$$\sup\limits_r \texpt{\set{h_{n,d}(X_r(\lambda))}^{p}}\leq
\frac{\Gamma(1-2p\left(n-\frac{d}{2}\right))}
{\Gamma(1-p\left(n-\frac{d}{2}\right))}\left(\mysum{k}{1}{j-1}\abs{\gamma_k}^2\right)^{pn}=C(n,d,p,\theta).$$

\bigskip
{\bf Case 2.} Suppose $n-\frac{d}{2}>0$.
Because of the inequality $0<n-\frac{d}{2}\leq \frac{1}{2}$, which is equivalent to
$0<2n-d\leq 1$, we conclude that $d=2n-1$, since $d$ and $n$ are positive integers. Therefore $n-\frac{d}{2}=\frac{1}{2}$. 
In this case, from the tools we developed in \S \ref{sec:improvement} we obtain that $0\leq h_{n,d}(X_r(\lambda))$ is bounded, up to some positive constant, by a finite sum containing terms of the form
\begin{equation}
\mysum{k}{1}{j-1}\abs{\gamma_k}^{2n}
\left(\spS{l_0,\alpha_r/2}\spS{l_1,\alpha_r/2}\right)^{\left(\frac{n}{2}-\frac{d}{4}\right)},
\end{equation}
for any two distinct numbers $l_0,l_1$ in $\set{1-(\incmt{1}{j}),\incmt{k}{k+1}}_{k=1}^{j-1}$.

Next, due to the fact that $\spS{l_i,\frac{\alpha_r}{2}}$, $i=0,1$ are independent and have law $l_i^{\frac{2}{\alpha_r}}S_{1,\alpha_r /2}$, $0<l_i<1$ and $n-\frac{d}{2}=\frac{1}{2}$, we conclude that
$\texpt{h_{n,d}(X_r(\lambda))^2}$ is bounded up to some positive constant by 
\begin{equation*}
\left(\texpt{S_{1,\frac{\alpha_r}{2}}^{\frac{1}{2}}}\mysum{k}{1}{j-1}\abs{\gamma_k}^{2n}\right)^2
\end{equation*}
We also know that $$\texpt{S_{1,\alpha_r/2}^{1/2}}=\frac{\Gamma(1-\frac{1}{\alpha_r})}{\Gamma(\frac{1}{2})}.$$

On the other hand, the function $\Gamma(z)$ is decreasing over (0,1). Next, we observe that each $\alpha_r$ satisfies
$$\frac{1}{3}\leq 1- \frac{1}{\alpha_r}\leq \frac{1}{2},$$ which yields
\begin{align*}
\texpt{S_{1,\alpha_r/2}^{1/2}}\leq\frac{\Gamma(\frac{1}{3})}{\Gamma(\frac{1}{2})}.
\end{align*}
For this, $$C(n,d,2,\theta)=C\left(\frac{\Gamma(\frac{1}{3})}{\Gamma(\frac{1}{2})}\mysum{k}{1}{j-1}\abs{\gamma_k}^{2n}\right)^2,$$ for some $C>0.$

\bigskip
{\bf Step 4.} $\lim\limits_{r\rightarrow+\infty}\texpt{h_{n,d}(X_r(\lambda))}=\texpt{h_{n,d}(X)}$.

This is a consequence of Steps 2, 3 and F3.

\bigskip
{\bf Step 5.} Notice that by H\"older's inequality and Step 3, we have proved that for some $p>1$, 
\begin{equation}\label{ineq.exp.weak}
\sup\limits_r \texpt{h_{n,d}(X_r(\lambda))}\leq\sup\limits_r\left(\texpt{h_{n,d}(X_r(\lambda))^p}\right)^{\frac{1}{p}}\leq \set{C(n,d,p,\theta)}^{\frac{1}{p}},
\end{equation}
where $\set{C(n,d,p,\theta)}^{\frac{1}{p}}>0$ is, indeed, a polynomial function in the variable $\theta$. Using the fact $V\in \mathcal{S}(\R^d)$ and the bounds in Step 3,  we have 
$$\int_{I_{j}}\iogRd{(j-1)}\set{C(n,d,p,\theta)}^{\frac{1}{p}}\abs{ \opV} d\theta_i d\lambda_i d\lambda_j<+\infty.$$ 

Therefore, by $\ref{ineq.exp.weak}$, Step 4, and dominated convergence theorem, we arrive at the desired result and this completes the proof of part (b). 

\section{Explicit form of some coefficients.}\label{sec:explicitcoeff}

In this section, we compute some coefficients explicitly and again show their  finiteness  by applying some basic inequalities arising from Lemmas \ref{finite.remainder} and \ref{Lemma1smallcase}.

For $n=0$ and any $j\geq1$,  we obtain by applying iterated times (\ref{Plancherel}) that
\begin{align*}
C_{0,j+1}^{(\alpha)}(V) &= \frac{ \const \texpt{ S_{1,\frac{\alpha}{2}}^{-d/2} } } { (2\pi)^{(j+1)d} } \lgint{j}{j}\sV{j}\pV{j}d\theta_i d\lambda_i d\lambda_j \nonumber \\ &= \frac{1}{(j+1)!}\ioRd V^{j+1}(\theta)d\theta, 
\end{align*}
where we have also used that $(4\pi)^{d/2}p_{1}^{(\alpha)}(0)=\texpt{S_{1,\frac{\alpha}{2}}^{-d/2}}$ and $\const= \frac{\pi^{d/2}}{ p_{1}^{(\alpha)}(0)}$.

The following Lemma will be useful to prove that the constants appearing in Corollary \eqref{maincor}  are strictly positive. Part of the following proof can be found in \cite{Davar}.  

\begin{lem}\label{Lem.Sub.estimate}
Given $0<\alpha<2$, there exists  $\alpf{N}>1$ such that 
\begin{equation*} \label{estimate for Sub}
\frac{1-e^{-v_{\alpha}}}{2}\leq P(1<S_{1,\frac{\alpha}{2}}< \alpf{N}) ,
\end{equation*}
where $\alpf{v}=(2-\alpha)\alpha^{\frac{\alpha}{2-\alpha}}2^{\frac{-2}{2-\alpha}}.$
\end{lem}

\begin{proof}
Let $a>0$ be fixed and  observe that $S_{1,\frac{\alpha}{2}}\leq a$ if and only if $e^{-\lambda S_{1,\frac{\alpha}{2}}}\geq e^{-\lambda a}$, for any  $\lambda>0$. Therefore, Chebyshev inequality tells us that

\[
P(S_{1,\frac{\alpha}{2}}\leq a)\leq \inf_{\lambda>0}e^{ \left(a\lambda-\lambda^{\frac{\alpha}{2}}\right)}
=e^{\left(-a^{\frac{\alpha}{\alpha-2}}v_{\alpha}\right)}.
\]

On the other hand, $\lim\limits_{n \to +\infty}P(S_{1,\frac{\alpha}{2}}< n)=1$ implies that given $\epsilon>0$, there exists a positive integer $N$ such that
\[
1-P(S_{1,\frac{\alpha}{2}}< n)\leq \epsilon, \,\,\,\, \textrm{for all $ n\geq N$.}
\]
It follows then that for $\epsilon= \frac{1-e^{-v_{\alpha}} }{2}$, there exists $\alpf{N}>1$ such that

\[
1-P(S_{1,\frac{\alpha}{2}}< \alpf{N}) \leq \frac{ 1-e^{-\alpf{v} }}{2}  \,\,\,\,\, \textrm{or} \hspace*{2mm}
P(S_{1,\frac{\alpha}{2}}< \alpf{N})\geq \frac{1+e^{-\alpf{v} } }{2}.
\]
Now use above facts with $a=1$, to obtain that 
\begin{eqnarray*}
P(1<S_{1,\frac{\alpha}{2}}<\alpf{N})&=&P(S_{1,\frac{\alpha}{2}}<\alpf{N})-P(S_{1,\frac{\alpha}{2}}\leq 1)\\
&\geq&  \frac{1+e^{-v_{\alpha}} }{2}-e^{-v_{\alpha}}\\
&=&\frac{1-e^{-v_{\alpha}}}{2}, 
\end{eqnarray*}
as desired. 
\end{proof}

\begin{rmk} Before proceeding, we give an explicit expression for $\alpf{N}$ when $\alpha=1$. Observe that $v_1= -\frac{1}{4}$. 
The $1/2$--subordinator $S$ can be expressed as the first hitting time for the standard one-dimensional  Brownian motion $\set{W_t}_{t\geq0}$. More precisely, 
\[
S_{t}=\inf\left\{s>0:W_s=\frac{t}{ \sqrt{2}}\right\}. 
\]
It is also known (See \cite[pp 23-24]{appleb} for details) that its density is given by
\[
\eta_{t}^{(1/2)}(s)= \frac{t}{2\sqrt{\pi}}s^{-3/2}e^{-t^2/4s}.
\]
Therefore, it is not difficult to see that for any $N>1$, 
\begin{align*}
P(1<S_{1,\frac{1}{2}}<N)=\frac{1}{2\sqrt{\pi}}\int_{1}^{N}s^{-3/2}e^{-1/4s}ds\geq 
\frac{e^{-1/4}}{2\sqrt{\pi}}\int_{1}^{N}s^{-3/2} ds= \frac{ e^{-1/4}} { \sqrt{\pi} } ( 1-N^{-1/2}).
\end{align*}
We can take then
$$\frac{ e^{-1/4} } { \sqrt{\pi} } ( 1-N_1^{-1/2} )=\frac{ 1-e^{-1/4} }{2}$$
or equivalently,
$N_1 = \set{1-\frac{\sqrt{\pi}}{2}(e^{1/4}-1)}^{-2}$ which is approximately $1.786$.
\end{rmk}
\bigskip
Let us now consider the case $n=1$ and $j=2$ . We recall  that 
\begin{equation}\label{S1casej=2,n=1}
\spS{1-(\incmt{1}{2})} + \spS{\incmt{1}{2}}=S_{1}, 
\end{equation}
provided $0<\lambda_2<\lambda_1<1$. In addition, $\spS{1-(\incmt{1}{2})}$ and $\spS{\incmt{1}{2}}$ are independent random variables. Then, it follows by Lemma \ref{lem.fin.rem} that
\begin{equation}
C_{1,2}^{(\alpha)}(V) = \frac{ \const K_1(d,\alpha) }{ (2\pi)^{d} }\langle -\Delta V, V\rangle=\frac{ \const K_1(d,\alpha)}{ (2\pi)^{d} }\int_{\Rd}|\nabla V(\theta)|^2\,d\theta, 
\end{equation}
where we have replaced $S_{1}$ by the left hand side of \eqref{S1casej=2,n=1} to obtain that
\begin{align*}\label{const for L.d.alp}
K_1(d,\alpha) &= \int_0^1 \int_0^{\lambda_1} 
\texpt{ \frac{ \spS{ 1-(\incmt{1}{2}) } \spS{\incmt{1}{2}} } 
{ (\spS{1-(\incmt{1}{2})} + \spS{\incmt{1}{2} })^{ 1+\frac{d}{2} } }}d\lambda_2 d\lambda_1 \nonumber \\
    &= \int_0^1 \int_0^{\lambda_1} \texpt{   \frac{ \spS{1-w} \spS{w} } { 
    ( \spS{1-w} + \spS{w} )^{  1+\frac{d}{2} } } }dw d\lambda_1.
\end{align*}

We now claim that  $K_1(d,\alpha)$ is both finite and strictly positive when either $d=1$ and $\frac{1}{2}<\alpha<2$, or $ d\geq2$ and $ 0<\alpha<2$ as follows. 

We start with  $d=1$ and $\frac{1}{2}<\alpha<2$. By Lemma \ref{Lemma1smallcase}, we obtain in this case that

\begin{equation*}
0\leq K_1(1,\alpha)\leq2^{-3/2}\texpt{S_{1,\alpha/2}^{1/4}}\int_0^1 \int_0^{\lambda_1}\set{(1-w)w}^{\frac{1}{2\alpha}}dw d\lambda_1.
\end{equation*}

On the other hand, when $d\geq2$ we have
\begin{eqnarray}\label{C_2^1 finite}
0 \leq K_1(d,\alpha)&\leq& \frac{1}{2} \int_0^1 \int_0^{\lambda_1 } \texpt{(\spS{1-w} + \spS{w})^{1-d/2}}dw d\lambda_1\\
& =& \frac{ \texpt{ S_{1,\alpha/2}^{1-d/2} } }{4}\nonumber, 
\end{eqnarray}
where we have used  the basic inequality 
 \begin{equation*}
\frac{ab}{(a+b)^{d/2+1}}=\frac{ab}{(a+b)^{d/2-1}(a+b)^2}\leq \frac{1}{2}(a+b)^{1-d/2},
\end{equation*}
valid for all $a,b>0$. The expectation in \eqref{C_2^1 finite} is finite for all $\alpha$ since
$1-\frac{d}{2}\leq0<\frac{\alpha}{2}$. Therefore $K_1(d,\alpha)$ is finite in the cases stated above.

Next, we prove that $K_1(d,\alpha)$ is strictly positive. We note that \eqref{CharF-LapS} implies that $S_{t}\eid t^{ \frac {2 }{\alpha} }S_{1}$. Therefore, we can write
\begin{equation*}
\spS{1-w}\eid (1-w)^{\frac{2}{\alpha}}X_1, \,\,\,\,\, \spS{w}\eid w^{\frac{2}{\alpha}}X_2, 
\end{equation*} 
where $X_1,X_2$ are independent copies of $S_{1}$.  That is, $X_1 \eid S_{1} \eid X_2$.  Thus, for any $0<w<1$, we have
\begin{align*}\label{ineq.exp.C_2^1}
\texpt{ \frac{ \spS{1-w} \spS{w} }{(\spS{1-w} + \spS{w})^{1+\frac{d}{2}} } } &=
\texpt{ \frac{ (1-w)^{\frac{2}{\alpha}}w^{2/\alpha}X_1X_2  }{ ( (1-w)^{2/\alpha}X_1  + w^{2/\alpha}X_2    ) ^{1 +\frac{d}{2}}}}
\nonumber \\
&\geq (1-w)^{  \frac{2}{\alpha} }w^{ \frac{2}{\alpha} }
 \texpt{  \frac{ X_1 X_2}{ ( X_1 +  X_2 )^{ 1+\frac{d}{2} }};1<X_1,X_2\ \leq \alpf{N} } \nonumber \\
 &\geq  \frac{ (1-w)^{ \frac{2}{\alpha} }w^{ \frac{2}{\alpha} } } {  (2\alpf{N} )^{ 1+ \frac{d}{2} } } P( 1< S_{1,\frac{\alpha}{2}} \leq \alpf{N} )^2. 
\end{align*}

From  Lemma \ref{Lem.Sub.estimate} and last inequality, we conclude that 
\begin{equation*}\label{C_2^1 positive}
K_1(d,\alpha)\geq  \frac{ (1-e^{-v_{\alpha}})^{2} }{4(2\alpf{N})^{ 1+ \frac{d}{2} }}  
 \int_0^1 \int_0^{\lambda_1}(1-w)^{  \frac{2}{\alpha} }w^{ \frac{2}{\alpha} }dw\,d\lambda_1>0.
\end{equation*}
Similarly, for either all $\alpha\in(1,2)$ and $d\geq2$,  or for $d\geq4$ and $\alpha \in (0,2)$,  we have 
\[C_{2,2}^{(\alpha)}(V)= \frac{ \const K_2(d,\alpha)}{ 2(2\pi)^{d}}\int_{\Rd}|\Delta V(\theta)|^2\,d\theta,
\] 
where
\[K_2(d,\alpha)= \int_0^1 \int_0^{\lambda_1}\texpt{ \frac{  ( \spS{1- (\incmt{1}{2} ) }\spS{\incmt{1}{2} } )^2 }{ 
    ( \spS{1-(\incmt{1}{2})} + \spS{\incmt{1}{2} })^{2+\frac{d}{2} } }}d\lambda_2\,   d\lambda_1.
\]
By applying the same argument as above, we have the following
\begin{enumerate}

\item[(i)] $\alpha\in(1,2)$ and $d\geq3$ or  for $d\geq4$ and $\alpha \in (0,2)$, we obtain
\[0<\frac{ (1-e^{-v_{\alpha}})^{2} }{4(2\alpf{N})^{ 2+ \frac{d}{2} }}  \int_0^1 \int_0^{\lambda_1}(1-w)^{  \frac{4}{\alpha} }w^{ \frac{4}{\alpha} }dw\,d\lambda_1
\leq K_2(d,\alpha)\leq \frac{ \texpt{S_{1,\frac{\alpha}{2}}^{2-d/2}} }{12}. 
 \]  
 \item[(ii)] $\alpha\in(1,2)$  and $d=2$, we have the same lower bound as in (i), but by Lemma \ref{smallcase} we obtain
 \[K_2(2,\alpha)\leq 2^{-3}\texpt{S_{1,\alpha/2}^{1/2}}
  \int_0^1 \int_0^{\lambda_1}\set{w(1-w)}^{\frac{1}{\alpha}}dwd\lambda_1\]

 \end{enumerate}
 
Likewise, it is not hard to prove that
 \begin{equation*}
 C_{1,3}^{(\alpha)}(V)=\frac{C_{d,\alpha}K_3(d,\alpha)}{(2\pi)^d}\int_{\R^d}V(\theta)\abs{\nabla V(\theta)}^2d\theta
  \end{equation*}
  where
  \begin{equation*}
  K_3(d,\alpha)=\int_{0}^{1}\int_{0}^{\lambda_1}\int_{0}^{\lambda_2}\texpt{
  \frac{ \spS{1-(\incmt{1}{3})}\spS{\incmt{1}{2}}+\spS{1-(\incmt{1}{3})}\spS{\incmt{2}{3}}  + \spS{\incmt{1}{2}}\spS{\incmt{2}{3}}}{(\spS{1-(\incmt{1}{3})}+\spS{\incmt{1}{2}} +\spS{\incmt{2}{3}})^{1+d/2}}}d\lambda_3 d\lambda_2 d\lambda_1,
  \end{equation*}
 is positive and finite provided that either $d\geq 2$ and $0<\alpha<2$ or
 $d=1$ and $\frac{1}{2}<\alpha<2$.
 
 For the rest of the paper, we will use the following notation for the constants given above,
\begin{align*}
\mathcal{M}_{d,\alpha}=\frac{C_{d,\alpha}K_3(d,\alpha)}{(2\pi)^d}
\,\,\, \textrm{and} \,\,\,
\mathcal{N}_{d,\alpha}=\frac{ \const K_2(d,\alpha)}{ 2(2\pi)^{d}}.
\end{align*}

\begin{rmk}
Based on the computations in \cite{Ba.Sab} and part (b) of Theorem \ref{mainthm}, we have 
under the conditions stated above that
\begin{align*}
\lim\limits_{\alpha\uparrow2}\mathcal{N}_{d,\alpha}=\frac{1}{120}\,\,\,and \,\,\,
\lim\limits_{\alpha\uparrow2}\mathcal{M}_{d,\alpha}=\frac{1}{12}.
\end{align*}

\end{rmk}

\section{Proof of Corollary \ref{maincor}.}\label{sec:proofcor}The proof uses a combination of Theorem \ref{mainthm} and Lemma \ref{smallcase}.
\bigskip

{\bf Case M=1}. 
\bigskip

\begin{enumerate}
\item[1)] When $d\geq2$, we invoke Theorem \ref{mainthm} with $J=3$. 
 We have in this case $1-\frac{d}{2}\leq 0<\frac{\alpha}{2}$. Therefore, we can consider any $\alpha$ on the right hand side of next expression, 
 
\begin{align}\label{case1}
\frac{\trace}{\palp(0)} + t\int_{\R^d} V(\theta)d\theta - \frac{t^2}{2}\int_{\R^d}V^2(\theta)d\theta + \frac{t^3}{3!}\int_{\R^d}V^3(\theta)d\theta = t^{\phi_4^{(\alpha)}(1)}R_4^{(\alpha)}(t),
\end{align}
where $\phi_4^{(\alpha)}(1)=\min\set{4,2+\frac{2}{\alpha}}$. Hence, {\bf (iii)} follows by noticing that
\begin{align*}
2+\frac{2}{\alpha}<4< 3+\frac{2}{\alpha}, \,\,\,\ when\,\, \alpha \in(1,2) \nonumber \\
4 \leq 2+\frac{2}{\alpha}< 3+\frac{2}{\alpha}, \,\,\,\ when \,\, \alpha \in (0,1]. 
\end{align*}

\item[2)]  For the case $d=1$, we use Lemma  \eqref{smallcase} which guarantees that \eqref{case1} is still true for $\frac{1}{2}<\alpha<1$. Thus {\bf (i)} holds. 
\end{enumerate}
\bigskip

{\bf Case M=2}.
\bigskip

 We apply Theorem \eqref{mainthm} with $J=4$ and
$2-\frac{d}{2}<\frac{\alpha}{2}$ to obtain
\begin{align}\label{case2}
&\frac{\trace}{\palp(0)} + t\int_{\R^d} V(\theta)d\theta -\frac{t^2}{2}\int_{\R^d}V^2(\theta)d\theta +
C_{1,2}^{(\alpha)}(V)t^{2+\frac{2}{\alpha}}  \nonumber \\
&+\frac{t^3}{3!}\int_{\R^d}V^3(\theta)d\theta - C_{1,3}^{(\alpha)}(V)t^{3+\frac{2}{\alpha}}
- \frac{t^4}{4!}\int_{\R^d}V^4(\theta)d\theta + C_{1,4}^{(\alpha)}(V)t^{4+\frac{2}{\alpha}} = t^{\phi_5^{(\alpha)}(2)}R_5^{(\alpha)}(t),
\end{align}
where $\phi_5^{(\alpha)}(2)=\min\set{5,2+\frac{2\cdot2}{\alpha}}$.
\begin{enumerate}
\bigskip
\item[1)] { \bf(iv)}, {\bf (v)} and {\bf (vi)} when $d\geq4$.

 We have $2-\frac{d}{2}\leq0<\frac{\alpha}{2}$ and then  any $\alpha$ can be considered. For $\alpha \in (0,1]$, we have $2\leq \frac{2}{\alpha},$ which implies that all the powers of $t$ containing $\alpha$ are larger  than 5 except $2+\frac{2}{\alpha}$. Comparing $2+\frac{2}{\alpha}$ with $5$ yields {\bf (v)} and {\bf (vi)}.
   Notice that {\bf (iv)} also follows since  all the power of $t$ containing $\alpha$ in \eqref{case2} are larger  than 4 (simply use the fact that $1<\frac{2}{\alpha}$) except  $2+\frac{2}{\alpha}$ which is less than 4 whenever $1<\alpha<2$.
    
   In addition, by Lemma \ref{smallcase} we obtain that (\ref{case2}) remains true in the following cases
  \item[2)] When $d=3$ and $\frac{1}{2}<\alpha <2$. Hence, {\bf (iv)} and {\bf (v)} holds by part 1).
\item[3)] When  $d=2$, $0<\alpha<1$. Then, {\bf (iv)} also holds for $d=2$.
\item[4)] When $d=1$,  $\frac{3}{2}<\alpha<2$. Thus, {\bf(ii)} holds. 
\end{enumerate}
This covers all the cases and completes the proof of Corollary \ref{maincor}.

\section{Explicit expansion for $\alpha=2/k$, $k\geq 2$ integer and $\alpha$ close to 2}\label{sec:Particular.expansions}
 
In this section we want to provide to the reader a better insight of our main theorem by finding a expansion formula of the trace when $\alpha= \frac{2}{k}$ with $k\geq2$ an integer and for values of $\alpha$ near 2, of course, under the condition that $2M-d <\alpha$, which is equivalent to $2M-d\leq 0$ when $0<\alpha\leq1$. We also give examples as an application to the results given below. We refer to the reader to the end of \S\ref{sec:explicitcoeff} for the definition of the constants $\mathcal{L}_{d,\alpha}$, $\mathcal{M}_{d,\alpha}$ and $\mathcal{N}_{d,\alpha}$.

\begin{thm}\label{expansion.alpha1}
Let $\alpha=1$ and $2M-d\leq0$. Then, for any $2\leq J\leq 2M+1$,
\begin{equation}
\frac{Tr(e^{-H_Vt}-e^{-H_1t})}{p_t^{(1)}(0)}+t\int_{\R^d}V(\theta)d\theta
-\mysum{l}{2}{J}t^{l}\left(\sum\limits_{\substack {2n+j=l, \\ j\geq 2}}(-1)^{n+j}C_{n,j}^{(1)}(V)\right)=\mathcal{O}(t^{J+1}),
\end{equation}
\end{thm}  
as $\tgo$.

\begin{proof} 
We apply Theorem \ref{2.mainthm} with $\alpha=1$ and $J+1 \leq 2M+2$,
 so that $\Phi_{J+1}^{(1)}(M)=\min\set{J+1,2M+2}=J+1$. Therefore,we obtain
\begin{align*}
\frac{Tr(e^{-H_Vt}-e^{-H_1t})}{p_t^{(1)}(0)} + t\ioRd V(\theta)d\theta 
- \sum\limits_{\substack{ j+2n\leq J, \\ 2\leq j\leq J,\,\, 0\leq n\leq M-1}}(-1)^{n+j}C_{n,j}^{(1)}(V)
t^{2n+j}= \mathcal{O}(t^{J+1}).
\end{align*}
as $\tgo$.

The following argument shows that under  the conditions given above we are not excluding any $n$ and $j$ such that $2n+j\leq J$.  
We observe that $2n+j=l$ for some $l \in\set{2,...,J}$ if and only if $n=\frac{l-j}{2}$. Then the larger $n$ can be is $\frac{J-2}{2}$. But $J\leq(2M+1)$, which yields
 $\frac{J-2}{2}\leq M-\frac{1}{2}$. Since $n$ is a positive integer, we conclude the larger that $n$ can be is in fact $M-1$. This implies that
 \begin{equation*}
 \sum\limits_{\substack{ j+2n\leq J, \\ 2\leq j\leq J,\,\, 0\leq n\leq M-1}}(-1)^{n+j}C_{n,j}^{(1)}(V)t^{2n+j}=\mysum{l}{2}{J}t^{l}\left(\sum\limits_{\substack {2n+j=l, \\  j\geq 2}}(-1)^{n+j}C_{n,j}^{(1)}(V)\right).
 \end{equation*}
\end{proof}

\begin{example}
When $M=2$ and $d\geq 4$, we have $2\cdot2-d\leq0<1$ . Then, 
Theorem \ref{expansion.alpha1} holds for any $2\leq J \leq 5$. Therefore, for the particular case $J=5$, we obtain
\begin{align}\label{examplealpha1}
\frac{Tr(e^{-H_Vt}-e^{-H_1t})}{p_t^{(1)}(0)} + &t\ioRd V(\theta)d\theta 
 - \frac{t^{2}}{2!}\ioRd V^2(\theta)d\theta \nonumber \\
 +&\frac{t^{3}}{3!}\ioRd V^3(\theta)d\theta 
-\frac{t^{4}}{4!}\left( \ioRd V^4(\theta)d\theta +4!{\mathcal L}_{d,1}\int_{\Rd}|\nabla V(\theta)|^2 d\theta\right) \nonumber \\ 
+&\frac{t^{5}}{5!}\left( \ioRd V^5(\theta)d\theta -5! \mathcal{M}_{d,1}
\int_{\R^d}V(\theta)\abs{\nabla V(\theta)}^2d\theta\right)=\mathcal{O}(t^6),
\end{align}
as $\tgo$.
Notice that part of this expansion is obtained by applying v) of Corollary \ref{maincor} to the specific case $\alpha=1$.
\end{example}

By mimicking the proof for the case $\alpha=1$, we conclude that
\begin{thm}
Let $\alpha = \frac{2}{k}$ with $k\geq3$ a positive integer. Assume also
$2M-d\leq 0$. Then, for any $2\leq J\leq 1+Mk$ we have 
\begin{equation}
\frac{Tr(e^{-H_Vt}-e^{-H_{\frac{2}{k}}t})}{p_t^{(\frac{2}{k})}(0)}+t\int_{\R^d}V(\theta)d\theta
-\mysum{l}{2}{J}t^l\left(\sum\limits_{\substack {kn+j=l,\\j \geq 2}}(-1)^{n+j}C_{n,j}^{(\frac{2}{k})}(V)\right)=\mathcal{O}(t^{J+1}),
\end{equation}
as $\tgo$.
\end{thm}

\begin{example} Consider $k=3$, $M=2$ and $d\geq 4$. Then, our main theorem holds for $2\leq J \leq 7$.  The particular case $J=5$ yields
\begin{align}\label{expansioncase2/3}
&\frac{Tr(e^{-H_Vt}-e^{-H_{\frac{2}{3}}t})}{p_t^{(\frac{2}{3})}(0)} + t\ioRd V(\theta)d\theta 
 - \frac{t^{2}}{2!}\ioRd V^2(\theta)d\theta +
\frac{t^{3}}{3!}\ioRd V^3(\theta)d\theta \nonumber\\
&-\frac{t^{4}}{4!}\ioRd V^4(\theta)d\theta 
+\frac{t^{5}}{5!}\left( \ioRd V^5(\theta)d\theta + 5!{\mathcal L}_{d,\frac{2}{3}}\int_{\Rd}|\nabla V(\theta)|^2 d\theta\right) =\mathcal{O}(t^6),
\end{align}
as $\tgo$.
\end{example}
\bigskip

 Let us consider for $J\geq2$ the following $J-1 \times J-1$ matrix which contains all the power of t with the form $\frac{2n}{\alpha}+j$ that may appear in the expansion of the trace, $A_{J}(\alpha)=\left( a_{r,s}\right)$ with $1\leq r\leq J-1$ and
 
 \begin{equation*}
 a_{r,s}=\left\{\begin{array}{cc}
                 r-s+2+\frac{2}{\alpha}(s-1) & \mbox{if $r \leq s$},\\
                  0                          & \mbox{otherwise}.
 \end{array}
 \right.
 \end{equation*}
 
In this matrix, $n=r-s+2$ and $j=s-1$. Observe that $n+j=r+1$. Thus, 
\begin{example} For $\alpha=1$,
\begin{equation*}
A_{6}(\alpha)= 
 \left( \begin{array}{ccccc}
 \bigskip
2 & . & . & .& .  \\
\bigskip
3 & 2+1\cdot\frac{2}{\alpha} & . & .& .  \\
\bigskip
4 & 3+1\cdot\frac{2}{\alpha} & 2+2\cdot \frac{2}{\alpha} & .&.\\
\bigskip
5 & 4+1\cdot\frac{2}{\alpha} & 3+2\cdot \frac{2}{\alpha} & 2+3\cdot \frac{2}{\alpha}&.\\
\bigskip
6 & 5+1\cdot\frac{2}{\alpha} & 4+2\cdot \frac{2}{\alpha} & 3+3\cdot \frac{2}{\alpha}& 2+4\cdot \frac{2}{\alpha} 
\end{array} \right)
=\left( \begin{array}{ccccc}
\bigskip
2 & . & . & .& .\\
\bigskip
3 & 4 & . & .& .  \\
\bigskip
4 & 5 & 6 & .&.\\
\bigskip
5 & 6 & 7& 8 &.\\
\bigskip
6 & 7 &8 & 9 & 10
\end{array} \right)
\end{equation*}
We have set the two matrices together to match  entry by entry. As an example, we conclude that there are two coefficients related to $t^5$. Namely, $C_{0,5}^{(1)}(V)$ and $C_{3,1}^{(1)}(V)$. The reader can verify this conclusion from \ref{examplealpha1}. 

Likewise, for $\alpha=\frac{2}{3}$ we obtain
\begin{equation*}
A_{7}(2/3)= 
 \left( \begin{array}{cccccc}
2 & . & . & .& .& .\\

3 & 5 & . & .& .& .\\

4 & 6 & 8 & .& .& .\\

5 & 7 & 9 & 11 & . &.\\

6 & 8 & 10 & 12 & 14 & . \\

7 & 9 & 11 & 13 & 15 & 17  
\end{array} \right)
\end{equation*}
 We can deduce then that the next two terms in the expansion given in \ref{expansioncase2/3}
 are
 \begin{equation*}
 -t^6\set{(-1)^6 C_{0,6}^{(\frac{2}{3})}(V) + (-1)^{3+1}C_{3,1}^{(\frac{2}{3})}(V)}
 -t^7\set{(-1)^7 C_{0,7}^{(\frac{2}{3})}(V) + (-1)^{4+1}C_{4,1}^{(\frac{2}{3})}(V)}.
 \end{equation*}
\end{example}

We point out that for any $0<\alpha<2$, we always have $2<3<2+\frac{2}{\alpha}$ which implies that the influence of the $\alpha$ in the expansion of the trace is expected to be seen in some place after the term  $C_{0,3}^{(\alpha)}(V)t^3$.

Notice that for every $J\geq 2$ we have
\begin{equation}
A_{J}(2)= 
 \left( \begin{array}{cccccc}
2 & . & . & .& .\\
 
3 & 3 & . & .& .  \\

4 & 4 & 4 & .&.\\

5 & 5 & 5& 5 &.\\
 
. & . &. &. & & .\\
J   &  J   &  J  &  J   &  ...  & J 
\end{array} \right)
\end{equation}

which says, for example, that in the expansion \eqref{Ba.Sab trace} there are three coefficients associated with $t^4$. Namely, $C_{0,4}^{(2)}(V),C_{1,3}^{(2)}(V)$ and $C_{2,2}^{(2)}(V)$. 
\begin{thm} 
Assume $J\geq 4$ and  $2(J-2)-d\leq 1$. Then, for all $\alpha\in (\frac{2(J-3)}{J-2},2)$ we have

\begin{align*}
\frac{\trace}{\palp(0)}+t\int_{\R^d} V(\theta)d\theta 
- \mysum{l}{2}{J-1}\sum\limits_{\substack {n+j=l,\\ j\geq2}}(-1)^{n+j}C_{n,j}^{(\alpha)}(V)
t^{\frac{2n}{\alpha}+j}= \mathcal{O}(t^{J}),
\end{align*}
as $\tgo$.
\end{thm}

\begin{proof}
In Theorem \ref{mainthm} we take $J=M+2\geq 4$ so that 
$$\Phi_{J+1}^{(\alpha)}(M)=\Phi_{J+1}^{(\alpha)}(J-2)= \min\set{J+1,2+\frac{2(J-2)}{\alpha}}>J$$
 since $\frac{2}{\alpha}>1$.
Now we want to choose $\alpha$ such that $a_{r,r}\leq a_{r+1,1}$ for $r\in\set{2,...,J-2}$. This last condition is equivalent to $\frac{2(J-3)}{J-2}<\alpha$ and implies that all the entries of $A_{J}(\alpha)$ are increasing for these $\alpha$`s. In other words, we have the following arrangement
\begin{align}\label{arrangement}
& \hspace*{4mm} 2\nonumber \\
&\leq 3\leq 2+\frac{2}{\alpha}\\
&\leq4\leq 3+\frac{2}{\alpha}\leq 2+2\cdot \frac{2}{\alpha} \nonumber\\
&.... \nonumber \\
&\leq J-1 \leq (J-2) +  \frac{2}{\alpha}\leq ...\leq 2+(J-3)\cdot \frac{2}{\alpha} <J\nonumber.
\end{align}
Then by  Theorem \ref{mainthm},
\begin{align}\label{exp.alphanear2}
\frac{\trace}{\palp(0)} + t\ioRd V(\theta)d\theta 
- \sum\limits_{\substack{ j+\frac{2n}{\alpha}< J \\ 2 \leq j \leq J-1,\,\,0\leq n \leq J-3}}(-1)^{n+j}C_{n,j}^{(1)}(V)
t^{\frac{2n}{\alpha}+j}= \nonumber \\
(-t)^{\Phi_{J+1}(J-2)}R_{J+1}(t)+\sum\limits_{ \substack {j+\frac{2n}{\alpha}\geq J \\ 2\leq j\leq J-1 ,\,\,0\leq n \leq J-3}}(-1)^{n+j}C_{n,j}^{(1)}(V)t^{\frac{2n}{\alpha}+j}.
\end{align}
We observe that the right hand side of \eqref{exp.alphanear2} is $\mathcal{O}(t^J)$ as $\tgo$, due to $\Phi_{J+1}(J-2)>J$.
On the other hand, it is easy to see by the definition of $A_{J}(\alpha)$ that the only powers of t satisfying $\frac{2n}{\alpha}+j<J$ are those in the arrangement given in \eqref{arrangement}. Therefore, due to this arrangement we can also rewrite the third term  in the left hand side of \eqref{exp.alphanear2} as follows
\begin{equation}
\sum\limits_{\substack{ j+\frac{2n}{\alpha}< J \\ 2\leq j\leq J-1 ,\,\,0\leq n \leq J-3}}(-1)^{n+j}C_{n,j}^{(\alpha)}(V)
t^{\frac{2n}{\alpha}+j}= \mysum{l}{2}{J-1}\sum\limits_{\substack {n+j=l,\\ j\geq2}}(-1)^{n+j}C_{n,j}^{(\alpha)}(V)
t^{\frac{2n}{\alpha}+j}.
\end{equation}
\end{proof}

\begin{example}
We take $J=4$ and $d\geq3$ so that $2\cdot2 -d\leq 1$. Then, according to last theorem for all $\alpha\in(1,2)$ we have
\begin{align*}
\frac{\trace}{\palp(0)} +&t\int_{\R^d}V(\theta)d\theta
-\frac{t^2}{2!}\int_{\R^d}V^2(\theta)d\theta \nonumber \\
+&\frac{t^3}{3!}\int_{\R^d}V^3(\theta)d\theta +{\mathcal L}_{d,\alpha}t^{2+\frac{2}{\alpha}}\int_{\R^d}\abs{\nabla V(\theta)}^2= \mathcal{O}(t^4),
\end{align*}
as $\tgo$. Notice that this result is already given in Corollary \ref{maincor}.

Let us now consider $J=5$ and $d\geq 5$. Then, for all $\alpha\in(\frac{4}{3},2)$ we obtain

\begin{align*}
\frac{\trace}{\palp(0)} + t\int_{\R^d}V(\theta)d\theta
-\frac{t^2}{2!}\int_{\R^d}V^2(\theta)d\theta
+\left(\frac{t^3}{3!}\int_{\R^d}V^3(\theta)d\theta +{\mathcal L}_{d,\alpha}t^{2+\frac{2}{\alpha}}\int_{\R^d}\abs{\nabla V(\theta)}^2\right) \nonumber \\
-\left(\frac{t^4}{4!}\int_{\R^d}V^4(\theta)d\theta + \mathcal{M}_{d,\alpha}
t^{3+\frac{2}{\alpha}}\int_{\R^d}V(\theta)\abs{\nabla V(\theta)}^2d\theta +
\mathcal{N}_{d,\alpha}t^{2+\frac{2\cdot 2}{\alpha}}
\int_{\Rd}|\Delta V(\theta)|^2\,d\theta
\right)= \mathcal{O}(t^5),
\end{align*}
as $\tgo$.
\end{example}

\section{Extension to $\alpha$--relativistic processes.}\label{sec:rel.proc}
 In this section we describe how to use  the tools developed in the preceding sections to obtain an asymptotic expansion for relativistic stable processes. 

Let $0<\alpha<2$ and $m\geq0$, we consider the function  $\phi_m(\lambda)= \set{ \lambda+m^{2/\alpha} }^{\alpha/2}-m$. As the case $m=0$ was already studied above, we assume for the rest of the paper that $m>0$. The function $\phi_m$ is known to be a Bernstein--function, that is $(-1)^n\phi_m^{(n)}(\lambda)\leq0$ for every $n\in \mathbb{N}$ and $\lambda>0$. Therefore, there exists a unique subordinator $\set{S_{t,m}}_{t\geq0}$ (see \cite[pp 89]{Bog})
 such that its Laplace transform is given by
\begin{equation*}\label{Laplace.relativistic}
\texpt{e^{-\lambda S_{t,m}}}= e^{-t\phi_m(\lambda)}.
\end{equation*}  

It is easy to see that the transition density $\eta_t^m(s)$ of $S_{t,m}$ satisfies the scaling 
\begin{equation*}
\eta_t^m(s)= e^{mt}\eta_t^{(\alpha/2)}(s)e^{-m^{2/\alpha}s},
\end{equation*}
which implies that for every $-\infty<\eta<\alpha/2$,
\begin{eqnarray}\label{rel.expt}
0<\texpt{S_{t,m}^{\eta}}&=& e^{mt}\texpt{S_{t,\frac{\alpha}{2}}^{\eta}e^{-m^{2/\alpha}S_{t,\frac{\alpha}{2}}}}\\
&\leq& e^{mt}t^{\frac{2\eta}{\alpha}}\texpt{S_{1,\alpha/2}^{\eta}}<\infty,\nonumber
\end{eqnarray}
where $S_{t,\frac{\alpha}{2}}$ is an $\alpha/2$--subordinator. 

The $\alpha$--stable relativistic process in $\R^d$ is defined as the subordinated Brownian motion process $Z_t^m=B_{2S_{t,m}}$, whose  characteristic function is given by
\begin{equation}
\texpt{e^{-i<\xi,Z_t^m>}}= e^{-t\phi_m(\abs{\xi}^2)}.
\end{equation}

If $p_t^{(\alpha,m)}(x)$ denote the transition density of $Z_t^m$ then exactly as in the preceding sections we have
\begin{equation*}
 p_t^{(\alpha,m)}(x)=\texpt{p_{S_{t,m}}^{(2)}(x)}.
 \end{equation*}
  Consequently, because of \eqref{rel.expt}, it follows that
\begin{eqnarray}\label{Rel.density.equation}
t^{d/\alpha}p_t^{(\alpha,m)}(0)&=& (4\pi)^{-d/2}\texpt{ S_{1,tm}^{-d/2} }\\
&=&(4\pi)^{-d/2}e^{mt}\texpt{S_{1,\frac{\alpha}{2}}^{-d/2}e^{-(tm)^{2/\alpha}S_{1,\frac{\alpha}{2}}}},\nonumber
\end{eqnarray}
which implies by monotone convergence theorem that
\begin{eqnarray*}\label{limit.rel}
\lim\limits_{\tgo}t^{d/\alpha}p_t^{(\alpha,m)}(0)e^{-mt}=p_1^{(\alpha)}(0),
\end{eqnarray*} and for all $0<t<1$,
\begin{equation}\label{ineq.rel.int}
t^{d/\alpha}p_t^{(\alpha,m)}(0)\geq(4\pi)^{-d/2}e^{-m}\texpt{S_{1,m}^{-d/2}}.
\end{equation}

We point out that in order to obtain the constants in the asymptotic expansion in Theorem \eqref{mainthm}, the self-similarity of the subordinator $S_{ t, \frac{\alpha}{2}}$, namely that 
 $t^{2/\alpha}S_{c,\frac{\alpha}{2}}\eid S_{ct,\frac{\alpha}{2}}$ and the scaling property of $\palp(x)$, were strongly used.  But, these two properties are not satisfied by $S_{t,m}$ and $p_t^{(\alpha,m)}(x)$. Indeed, we have
 \begin{equation*}
t^{2/\alpha}S_{c,tm}\eid S_{ct,m},\,\,\,\textrm{and} \,\,\,p_{t}^{(\alpha,m)}(x)
 =m^{d/\alpha}
p_{mt}^{(\alpha,1)}
 (m^{1/\alpha}x).
 \end{equation*} 

Despite of this, we can still obtain a version of Theorem \ref{mainthm} where the coefficients $C_{n,j}^{(\alpha)}(V)$ are replaced by time dependent functions. 

We now define the following linear operators  $H_{\alpha,m}= \phi_m(-\Delta)$ and $H_{V}= H_{\alpha,m}+V$. In \cite{Hiro} it is  proved that for any Bernstein function  the Feymann Kac formula holds. As a result, the heat kernel of $H_V$ can also be written as
\begin{equation*}
p_t^{H_V}(x,y)= p_t^{(\alpha,m)}(x,y)E_{x,y}^t\left[e^{-\int_{0}^{t}V(Z_s^m)ds}\right].
\end{equation*}

 We now proceed to mention which of the results in the above sections still hold true for these operators.  It is easy to check that all the results in section $\S\ref{sec:trace&Fourier}$ and $\S\ref{sec:boundedcoeff.rem}
$ are still true when we replace $\abs{\xi}^{\alpha}$ by $\phi_m(\abs{\xi}^2)$, $p_t^{(\alpha)}$ by $p_t^{(\alpha,m)}$ and using the fact  that $\wh{\phi_m(-\Delta)f}(\xi)=\phi_m(\abs{\xi}^2)\wh{f}(\xi)$. Furthermore, by replacing the random variables $S_{\lambda}$ by $S_{\lambda,tm}$ and $S^*_{\lambda}$ by $S^*_{\lambda,tm}$, we see that all results in  \S\ref{sec:trace-sub} also remain true with the following modifications. 
Set, 
\begin{equation*}
F_j^{(m)}(t\lambda,\xi,\theta)=e^{-t(1-\lambda_1)\phi_{m}(|\xi|^{2})-t\mysum{k}{1}{j-1}(
\lambda_k - 
\lambda_{k+1}) \phi_m(|\xi - \mysum{i}{1}{k}\theta_i|^{2})},
\end{equation*}
and 
\begin{equation*}
L_j(tm,\lambda,\theta)= \mysum{k}{1}{j-1}\spS{\incmt{k}{k+1},tm}|\gamma_k|^2 -\frac{1}{S_{1,tm}}\left |\mysum{k}{1}{j-1}\spS{\incmt{k}{k+1},tm}\gamma_k\right|^2,
\end{equation*}
where $\lambda=(\lambda_1,...,\lambda_j)$ satisfies \eqref{lamb.ineq} and $\theta=(\theta_1,..., \theta_{j-1})$. Then
\begin{equation*}
\int_{\R^d} F_j^{(m)}(t\lambda,\xi,\theta)e^{-t\lambda_j\phi_m(\abs{\xi}^2)}d\xi= \pi^{d/2}
t^{-d/\alpha} 
\texpt{S_{1,tm}^{-d/2}e^{-t^{2/\alpha}L_j(tm,\lambda,\theta)}}.
\end{equation*} We recall that we cannot express the right hand side of the last equality in terms of $p_{t}^{(\alpha,m)}(0)$ because this density fails to satisfy the scaling  property. 
We now state the analogue of Proposition \eqref{Fourier-trace-formula} for the new linear operators defined above.

\begin{prop}
Given $J\geq2$, we have
\begin{align*}
&\frac{Tr{e^{-tH_V}-e^{-tH_{\alpha,m}}}}{p_t^{(\alpha,m)}(0)}=
-t\ioRd V(\theta)d\theta \nonumber \\
&+ \mysum{j}{2}{J}
\frac{\pi^{d/2}(-t)^{j}}{(2\pi)^{jd}t^{d/\alpha}p_t^{(\alpha,m)}(0)}\lgint{j-1}{(j-1)}
\texpt{S_{1,tm}^{-d/2}e^{-t^{2/\alpha}L_j(tm,\lambda,\theta)}}\\
&\times\opV d\theta_i d\lambda_i d\lambda_j + (-t)^{J+1}R_{J+1}(t), 
\end{align*}
where $R_{J+1}(t)$ is a bounded function  on $(0,1)$.
\end{prop}

From this lemma, it is clear that with the Taylor-expansion of the exponential function we obtain time depending functions instead of coefficients as in \S\ref{sec:coeffic.sect}.
Now, it is easy to see by using the expansion \eqref{Tay.exp} of order 2 that when
$1<\frac{\alpha+d}{2}$ the following equality holds
\begin{equation*}
\frac{\texpt{S_{1,tm}^{-d/2}e^{-t^{2/\alpha}L_j(tm,\lambda,\theta)}}}{t^{d/\alpha}p_t^{(\alpha,m)}(0)}=(4\pi)^{d/2} - t^{2/\alpha}R_{1,j}(t),
\end{equation*}
where 
\begin{equation*}
R_{1,j}(t)= \frac{\texpt{S_{1,tm}^{-d/2}L_j(tm,\lambda,\theta)e^{-\beta_1}}}{t^{d/\alpha}p_t^{(\alpha,m)}(0)}, 
\end{equation*} 
for some nonnegative random function $\beta_1= \beta_1(tm,\lambda,\theta)$ and this satisfies, according to \eqref{ineq.rel.int} and (\ref{rel.expt}), the estimate 
\begin{equation*}
\abs{R_{1,j}(t)}\leq \frac{(4\pi)^{d/2}e^{2m}\texpt{S_{1,\frac{\alpha}{2}}^{-d/2+1}}}{\texpt{S_{1,m}^{-d/2}}}\mysum{k}{1}{j-1}\abs{\gamma_k}^2,
\end{equation*}  for all $0<t<1$. 
With this last estimate, by mimicking the proof of Theorem \eqref{mainthm} ,we have that

\begin{align}\label{trace.rel}
\frac{Tr(e^{-H_Vt}-e^{-H_{m}t} )}{p_t^{(\alpha,m)}(0)}+t\int_{\R^d} V(\theta)d\theta - \frac{t^{2}}{2!} 
\int_{\R^d} V^2(\theta)d\theta 
+\frac{t^{3}}{3!}\int_{\R^d} V^3(\theta)d\theta=
t^{\Phi_4^{(\alpha)}(1)}R^{(m)}_{4}(t),
\end{align} where
$\Phi_4^{(\alpha)}(1)=\min\set{4,2+\frac{2}{\alpha}}$, and $R_{4}^{(m)}(t)$ is a bounded function of $t$ over $(0,1)$. 
We can also redo the proof of Lemma \ref{smallcase} to make sure that \eqref{trace.rel} remains true when $d=1$, $M=1$, and $\frac{1}{2}< \alpha< 2$. For this, we only require the following inequality when $0<t<1$, that follows again from (\ref{rel.expt}),
\begin{eqnarray*}
\texpt{S_{l,tm}^{M/2-d/4}}&=&l^{M/2-d/4}\texpt{S_{1,ltm}^{M/2-d/4}}\\
&\leq&
l^{M/2-d/4}e^{lm}\texpt{S_{1,\frac{\alpha}{2}}^{M/2-d/4}}.
\end{eqnarray*}
From this we conclude 

\begin{cor} 
\begin{align*} 
\frac{Tr(e^{-H_Vt}-e^{-tH_{\alpha,m}} )}{p_t^{(\alpha,m)}(0)}+t\int_{\R^d} V(\theta)d\theta - \frac{t^{2}}{2!}\int_{\R^d} V^2(\theta)d\theta
+ \frac{t^{3}}{3!}\int_{\R^d} V^3(\theta)d\theta= 
\nonumber \\
\left\{\begin{array}{ccc}
       {\mathcal O}( t^{2 + \frac{2}{\alpha}}), &  \mbox{ if $\alpha\in(1,2)$ and $d\geq1$} , \nonumber\\
       {\mathcal O}(t^4), & \mbox{ if $\alpha\in(\frac{1}{2},1]$ and $d\geq1$}, \nonumber \\
       {\mathcal O}(t^4), & \mbox{ if $\alpha\in(0,\frac{1}{2}]$ and $d\geq2$},
       \end{array}   
\right. 
\end{align*} 
as $\tgo$.
\end{cor}
This extends the result in Ba\~{n}uelos and Yildirim \cite{Ba.Sel} where the second term is computed.

\section{Extension to mixed--stable processes.}\label{sec:mixed.proc}
In this section we assume that  $0<\alpha<\beta<2$, $a>0$ and only extend Theorem \ref{mainthm} for the case $1<\frac{\alpha+d}{2}$.  In addition, we analyse the cases $d=1$,
$1<\alpha<2$ and $d=2$, $0<\alpha <\beta<2$.
We point out that the  constants  $C_{a,d,\alpha,\beta}$  below may change from line to line.
Consider the Bernstein function $\phi_a(\lambda)= \lambda^{\frac{\alpha}{2}}+a\lambda^{\frac{\beta}{2}}$ and let $\set{S_{t,a}}_{t>0}$ be the subordinator associated to $\phi_a$, with Laplace transform given by
\begin{equation*}
\texpt{e^{-\lambda S_{t,a}}}=e^{-t\phi_a(\lambda)}.
\end{equation*}

It is a standard fact that the above subordinator can be written as 
\begin{align}\label{sub.ineq.mixed stable}
S_{t,a}=S_{t,\frac{\alpha}{2}}+a^{\frac{2}{\beta}}S_{t,\frac{\beta}{2}},
\end{align} 
where the subordinators in right hand side of the last equality are independent processes. 

Notice that, by independence, we have
\begin{equation}\label{sub.mixed.eid}
S_{t,a}\eid t^{\frac{2}{\alpha}}S_{1,\frac{\alpha}{2}}+ (at)^{\frac{2}{\beta}}S_{1,\frac{\beta}{2}},
\end{equation} and by \eqref{sub.ineq.mixed stable}, it also follows that
\begin{equation*}
(1+a^{\frac{2}{\beta}})\min\set{ S_{t,\frac{\alpha}{2}},S_{t,\frac{\beta}{2}}}\leq S_{t,a}\leq(1+a^{\frac{2}{\beta}})
\max\set{ S_{t,\frac{\alpha}{2}},S_{t,\frac{\beta}{2}}},
\end{equation*}
which implies for any $\eta<\frac{\alpha}{2}$, 
\begin{align}\label{ineq.exp.mixed.stable}
0<\texpt{S_{t,a}^{\eta}}\leq &2(1+a^{\frac{2}{\beta}})^{\eta}\max\left\{ \texpt{S_{t,\frac{\alpha}{2}}^{\eta}},
\texpt{S_{t,\frac{\beta}{2}}^{\eta}}\right\} \nonumber \\
 \leq &2(1+a^{\frac{2}{\beta}})^{\eta}\max\left\{t^{\frac{2\eta}{\alpha}} \texpt{S_{1,\frac{\alpha}{2}}^{\eta}},
t^{\frac{2\eta}{\beta}}\texpt{S_{1,\frac{\beta}{2}}^{\eta}}\right\}
<\infty.
\end{align}

The $d$--dimensional process $Y_t^a=B_{2S_{t,a}}$ is known as the mixed $(\alpha,\beta)$--stable process, with characteristic function given by
$$\texpt{e^{-i<\xi,Y_t^a>}}= e^{-t\phi_a(\abs{\xi}^2)}.$$
 Properties for the heat kernels and Green functions for these processes have been studied by several people, see for example Chen and Kumagai \cite{CheKum}, Chen, Kim and Song \cite{CheKimSon} or Jakubowski and Szczypkowski \cite{JakSzc}. Using some of these results, Ba\~{n}uelos and Yildirim \cite{Ba.Sel} were able to compute the second coefficient.  In this action we show how to go further than the results in \cite{Ba.Sel} using the techniques of this paper.  

Denote the transition density of $Y_t^a$ by  $p_t^{(a)}(x)$, then as before we have
$$p_t^{(a)}(x)=\texpt{p_{S_{t,a}}^{(2)}(x)}.$$
Therefore, by \eqref{sub.mixed.eid}, the following holds
\begin{eqnarray*}
p_t^{(a)}(0)&=&(4\pi)^{-d/2}\texpt{S_{t,a}^{-d/2}}\\
&=&
(4\pi)^{-d/2}\texpt{\left( t^{\frac{2}{\alpha}}S_{1,\frac{\alpha}{2}}+ (at)^{\frac{2}{\beta}}S_{1,\frac{\beta}{2}}\right)^{-d/2}}.
\end{eqnarray*}
It is also proved in  \cite{CheKimSon}, that there exists a constant $C_{d,\alpha,\beta}$ such that
$$C_{d,\alpha,\beta}^{-1}f_t^{a}(x)\leq p_t^{(a)}(x) \leq C_{d,\alpha,\beta}f_t^{a}(x)$$
where 
$$f_t^{a}(x)=\min\left\{ \min\left\{t^{-d/\alpha},(at)^{-d/\beta}\right\},\frac{t}{\abs{x}^{d+\alpha}}+\frac{at}{\abs{x}^{d+\beta}}\right\}. $$
In particular, it follows that there exist a constant $C_{a,d,\alpha,\beta}>0$ such that 
\begin{equation}\label{lower.bound.mixed.density}
p_t^{(a)}(0)\geq C_{a,d,\alpha,\beta}t^{-d/\beta},
\end{equation}
provided  $0<t<\min\set{1,a^{\frac{\alpha}{\beta-\alpha}}}$.

Set
\begin{align*}
F^{(a)}(t\lambda,\xi,\theta)&=e^{-t(1-\lambda_1)\phi_a(|\xi|^{2})-t\mysum{k}{1}{j-1}(\lambda_k - \lambda_{k+1}) \phi_a(|\xi - \mysum{i}{1}{k}\theta_i|^{2})}, \\\\
A_0(t,\lambda)&=t^{\frac{2}{\alpha}}S^*_{1-(\incmt{1}{j}),\frac{\alpha}{2}} + (at)^{\frac{2}{\beta}}S^*_{1-(\incmt{1}{j}),\frac{\beta}{2}},\\\\
A_k(t,\lambda)&=t^{\frac{2}{\alpha}}S^*_{\incmt{k}{k+1},\frac{\alpha}{2}} + (at)^{\frac{2}{\beta}}S^*_{\incmt{k}{k+1},\frac{\beta}{2}},\\\\
L_{j}^{(a)}(t,\lambda,\theta)&=\mysum{k}{1}{j-1}A_k(t,\lambda)\abs{\gamma_k}^2-\frac{\abs{\mysum{k}{1}{j-1}A_k(t,\lambda)\gamma_k}^2}{t^{\frac{2}{\alpha}}S_{1,\frac{\alpha}{2}} + (at)^{\frac{2}{\beta}}S_{1,\frac{\beta}{2}}}.
\end{align*}
Then, we obtain by using the tools of $\S 5$ that,
\begin{equation}
\ioRd F^{(a)}(t\lambda,\xi,\theta)e^{-t\lambda_j\phi(\abs{\xi}^2)}d\xi=\pi^{-d/2}
\texpt{\left(t^{\frac{2}{\alpha}}S_{1,\frac{\alpha}{2}} + (at)^{\frac{2}{\beta}}S_{1,\frac{\beta}{2}}\right)^{-d/2}e^{-L_{j}^{(a)}(t,\lambda,\theta)}}.
\end{equation}
We also observe that 
\begin{equation}
\mysum{k}{0}{j-1}A_k(t,\lambda)= t^{\frac{2}{\alpha}}S_{1,\frac{\alpha}{2}} + (at)^{\frac{2}{\beta}}S_{1,\frac{\beta}{2}}.
\end{equation}
With this equality, we can imitate the proof of Lemma \eqref{lem.fin.rem} along with the Taylor expansion to arrive at 
\begin{equation}
\frac{\texpt{\left(t^{\frac{2}{\alpha}}S_{1,\frac{\alpha}{2}} + (at)^{\frac{2}{\beta}}S_{1,\frac{\beta}{2}}\right)^{-d/2}e^{-L_{j}^{(a)}(t,\lambda,\theta)}}}{p_t^{(a)}(0)}=
(4\pi)^{d/2} - R_{d,1}^{(1)}(t),
\end{equation}
 with
\begin{align}\label{ineq.rem}
\abs{ R_{d,1}^{(1)}(t)}&\leq \frac{\texpt{\left(t^{\frac{2}{\alpha}}S_{1,\frac{\alpha}{2}} + (at)^{\frac{2}{\beta}}S_{1,\frac{\beta}{2}}\right)^{-d/2+1}}}{p_t^{(a)}(0)}\mysum{k}{1}{j-1}\abs{\gamma_k}^2 \nonumber \\
&\leq C_{a,d,\alpha,\beta}\mysum{k}{1}{j-1}\abs{\gamma_k}^2t^{\frac{d}{\beta}}\max\left\{t^{\frac{2}{\alpha}\left(1-\frac{d}{2}\right)} \texpt{S_{1,\frac{\alpha}{2}}^{-d/2+1}},
t^{\frac{2}{\beta}\left(1-\frac{d}{2}\right)}\texpt{S_{1,\frac{\beta}{2}}^{-d/2+1}}\right\},
\end{align}
where we have used \eqref{lower.bound.mixed.density} and \eqref{ineq.exp.mixed.stable}.

Let us set $H_{a}= \phi_{a}(-\Delta)$ and $H_V=H_a+V$.
With this notation and the above bounds we obtain (by mimicking the proof of Theorem \eqref{mainthm}) the following result.

\begin{thm} Assume d=1 and $1<\alpha<\beta<2$. Then,
\begin{align*}
\frac{Tr(e^{-H_Vt}-e^{-tH_a})}{p_t^{(a)}(0)} + t\int_{\R}V(\theta)d\theta  
-\frac{t^2}{2!}\int_{\R}V^2(\theta)d\theta + \frac{t^3}{3!}\int_{\R}V^3(\theta)d\theta=
{\mathcal O}(t^{2+\frac{2}{\beta}}),
\end{align*}
as $\tgo$.
\end{thm}

\begin{thm} Assume $d\geq2$ , $0<\alpha<\beta<2$ and $2+\frac{2}{\alpha} -d\left(\frac{1}{\alpha}-\frac{1}{\beta}\right)\geq 0$.  Then,
\begin{align*}
\frac{Tr(e^{-H_Vt}-e^{-tH_a})}{p_t^{(a)}(0)} + t\int_{\R^d}V(\theta)d\theta  
-\frac{t^2}{2!}\int_{\R^d}V^2(\theta)d\theta + \frac{t^3}{3!}\int_{\R^d}V^3(\theta)d\theta=
t^{\Phi_{d}(\alpha,\beta)}R(t),
\end{align*}
for some bounded function $R(t)$, $t\in \left(0,\min\left\{a^{\frac{\alpha}{\beta-\alpha}},1\right\}\right)$.   Here, 
\[ \Phi_{d}(\alpha,\beta) 
= \min \left\{ 4,2+\frac{2}{\alpha}- d \left( \frac{1}{\alpha}-
\frac{1}{\beta} \right)\right\}.\]
\end{thm} 

As an example, observe that 
\[\Phi_2(\alpha,\beta)=\min\set{4, 2+\frac{2}{\beta}}.\]  This gives 

\begin{cor} For d=2, we have
\begin{align*}
\frac{Tr(e^{-H_Vt}-e^{-tH_a})}{p_t^{(a)}(0)} + t\int_{\R^2}V(\theta)d\theta  
-\frac{t^2}{2!}\int_{\R^2}V^2(\theta)d\theta + \frac{t^2}{3!}\int_{\R^2}V^3(\theta)d\theta = 
\nonumber \\
\left\{\begin{array}{cc}
       {\mathcal O}( t^{2 + \frac{2}{\beta}}), &  \mbox{if  $1\leq\beta$ and $\alpha<\beta$}, \nonumber\\
       {\mathcal O}(t^4), & \mbox{if  $0<\alpha<\beta<1$ }. \nonumber \\
       \end{array}   
\right. 
\end{align*}
\end{cor}

Finally, we remark  that the condition  $2+\frac{2}{\alpha}>d\left(\frac{1}{\alpha}-\frac{1}{\beta}\right)$ comes from \eqref{ineq.rem} in order to obtain the boundedness of $R(t)$
over the interval stated above. 
\bigskip

{\bf Acknowledgements}: I am grateful  to my supervisor,  Professor Rodrigo Ba\~nuelos, for his valuable guidance, advice and time while preparing this paper.


\begin{thebibliography}{20}  
\bibitem{appleb} D. Applebaum, {\it L\'{e}vy Processes and Stochastic Calculus} (second edition), Cambridge University Press (2009).

\bibitem{AreSch} W. Arendt and W.P. Schleich(eds.), {\it Mathematical Analysis of Evolution, Information, and Complexity},  Wiley-VCH Verlag GmbH \& Co. KGaA, Weinheim, Germany, (2009). 

\bibitem{Ba.Sel}
 R. Ba\~nuelos and S. Yildirim, Heat trace of non-local operators, 2011 preprint, http://arxiv.org/abs/1201.2171
 
\bibitem{Ba.Kul.Siu}
 R. Ba\~{n}uelos, T. Kulczycki, B. Siudeja, On the heat trace of symmetry stables processes on Lipschitz domains, {J. Funct. Anal.} {\bf 257}  (2009), 3329-3352.
 
\bibitem{Ba.Kul}
 R. Ba\~{n}uelos, T. Kulczycki, Trace estimates for stable processes, {\it Prob.Theory Relat. Fields.} {\bf 142}  (2008), 313-338.

\bibitem{Ba.Sab}
R. Ba\~{n}uelos and A. S\'{a} Barreto, On the heat trace of Schr\"{o}dinger operators, {\it Comm in Partial Differential equations.} {\bf 20} (1995),  2153-2164. 

\bibitem{Billingsley}
P.Billingsley,{\it Probability and Measure}(Third edition), Wiley series in probability and mathematical statistics.(1995),327-352, 378-386.

\bibitem{Bog}
 K. Bogdan, {\it Potential Analysis of Stable Processes and its extensions,} Lecture notes in mathematics {\bf 1980}, Springer-Verlag (2009).
 
\bibitem{CheKimSon}
 Z.-Q. Chen, P. Kim and R. Song, Dirichlet heat kernel estimates for $\Delta^{\alpha/2} +\Delta^{\beta/2}$ , (To appear in Ill J. Math.)
 
\bibitem{CheKum} Z.-Q. Chen and T. Kumagai, Heat kernel estimates for jump processes of mixed types on metric measure spaces, {\it Probab. Theory Relat. Fields} {\bf 140} (2008), 277--317.

\bibitem{dev} Y.Colin de Verdi\`ere, Une formule de trace pour l'uperateur de Schr\"odinger dans $\R^3$, 
 {\it Ann. Scient. \`Ec. Norm. Sup.} {\bf 14} (1981), 27--39. 
 
\bibitem{KirHez} K. Datchev and H. Hezari, Inverse problems in spectral geometry, arXiv:1108.5755v2, 2012. 

\bibitem{Davies}
 E.B. Davies, {\it Heat kernels and spectral theory}, Cambridge University press {\bf 92} (1989).

\bibitem{Don}
 H. Donnelly, Compactness of isospectral potentials, {\it Trans. Amer. Math. Soc.} {\bf 357} (2005),1717-1730.
 
\bibitem{Davar} D. Khoshnevisan, {\it Topics in probability: L\'{e}vy Processes}. Lecture notes. \href{http://www.math.utah.edu/~davar/ps-pdf-files/Levy.pdf}{http://www.math.utah.edu/$\sim$davar/ps-pdf-files/Levy.pdf.}
 
 \bibitem{Hiro} F. Hiroshima,T. Ichinose and J. L\"orinczi,  {\it Path Integral Representation for 
Schr\"odinger Operators with Bernstein Functions of the Laplacian,}  (arXiv:0906.0103v4), 2010.  

\bibitem{JakSzc} T. Jakubowski and K. Szczypkowski, {Estimates of gradient perturbations series}, 
{\bf (2011 preprint--arXiv:1110.1672v1)} 

\bibitem{Lieb}
 E.H. Lieb, Calculation of exchange second viral coefficient of a hard-sphere gas by path integrals,  {\it Journal of Mathematical Physics} {\bf 8} (1967), 43-52.
 
 \bibitem{MckMoe} H.P. McKean and P. van Moerbeke, The spectrum of Hill's equation, {\it Inventiones Math.} {\bf 30} (1975), 217--274. 
 
  \bibitem{Penrose}M.D.Penrose, O.Penrose, G. Stell, Sticky spheres in quantum mechanics, {\it Reviews of Mathematical Physics 6} (1994). Also {\it The states of matter} (volume dedicated to E.H. Lieb), edited by M. Aizenman and H. Araki, {\it World Scientific}.

\bibitem{Reed}
M.Reed and B.Simon, {\it Functional Analysis}, {Academic Press} (1972).Volume I.249-270.

\bibitem{Keni}
 K. Sato, {\it L\'{e}vy Processes and infinitely divisible distributions,} Cambridge Studies in Advanced Mathematics {\bf 68}, Cambridge University Press, (1999). 
 
 \bibitem{Pascal} P. Sebah, X. Gourdon. {\it Introduction to the Gamma Function}.
 \url{http://www.frm.utn.edu.ar/analisisdsys/material/funcion_gamma.pdf}
\href{http://www.frm.utn.edu.ar/analisisdsys/material/funcion_gamma.pdf}.2002
 
\bibitem{Barry}
 B. Simon, Schr\"{o}dinger semigroups, {\it Bulletin of the AMS}  {\bf 7} (1982),  447-526 .
 \bibitem{vanden1} M. van den Berg,  On the trace of the difference of Schr{\"o}dinger heat semigroups, {\it Proceedings of the Royal Society of Edinburg} {\bf 119A } (1991), 169-175.
 
\end{thebibliography}
\end{document}